\documentclass[12pt,a4paper,oneside,reqno]{amsart}
\usepackage{fancyhdr}
\usepackage{tikz}
\tikzset{wellposed/.style={cyan}}
\tikzset{illposed/.style={magenta}}
\usepackage{pgfplots}
\definecolor{axiscolor}{rgb}{0,0,0}
\pgfplotsset{compat=1.18}
\usepackage{amssymb,mathrsfs,amsmath,amsthm}
\usepackage[numbers,sort&compress]{natbib}
\usepackage{float}
\usepackage{authblk}
\usepackage{url}
\usepackage[numbers]{natbib}
\usepackage{setspace}
\usepackage{mathtools}
\usepackage{comment}
\usepackage[inline]{enumitem}
\usepackage[
  colorlinks=true,
  linkcolor=blue,
  citecolor=blue,
  urlcolor=blue,
  pagebackref,
  pdfauthor={Your Name},
  pdftitle={Your Document Title},
  pdfkeywords={keyword1, keyword2, keyword3}
]{hyperref}
\usepackage{booktabs}
\usepackage{comment}
\usepackage{hyperref}
\usepackage[margin=1in]{geometry}

\newtheorem{theoremA}{Theorem}

\newtheorem{thm}{Theorem}[section]
\newtheorem{cor}[thm]{Corollary}
\newtheorem{lem}[thm]{Lemma}
\newtheorem{prop}[thm]{Proposition}
\theoremstyle{definition}

\newtheoremstyle{boldremark}
  {10pt}  
  {10pt}   
  {}      
  {}       
  {\bfseries} 
  {.}      
  { }     
  {}  
\theoremstyle{boldremark}
\newtheorem{Remark}[thm]{\textbf{Remark}}
\numberwithin{equation}{section}
\newenvironment{mathclass}
  {Mathematics Subject Classification (2010):}
 
\newcommand{\R}{\mathbb{R}}
\newcommand{\C}{\mathbb{C}}
\newcommand{\N}{\mathbb{N}}

\renewcommand{\keywords}[1]{%
  \par\noindent
  Keywords: #1
  \par
}

\def\R {\mathbb{R}}

\newcommand{\be}{\begin{equation}}
\newcommand{\ee}{\end{equation}}
\newcommand{\bea}{\begin{eqnarray}}
\newcommand{\eea}{\end{eqnarray}}
\newcommand{\Bea}{\begin{eqnarray*}}
\newcommand{\Eea}{\end{eqnarray*}}
\newcommand{\bt}{\begin{Theorem}}
\newcommand{\et}{\end{Theorem}}
\newcommand{\bpr}{\begin{Proposition}}
\newcommand{\epr}{\end{Proposition}}

\newcommand{\bl}{\begin{Lemma}}
\newcommand{\el}{\end{Lemma}}
\newcommand{\bi}{\begin{itemize}}
\newcommand{\ei}{\end{itemize}}

\newtheorem{Definition}{Definition}[section]
\newtheorem{Theorem}[Definition]{Theorem}
\newtheorem{Lemma}[Definition]{Lemma}
\newtheorem{Proposition}[Definition]{Proposition}

\title
{ Inhomogeneous NLS  in negative  Sobolev spaces \\ via Fourier-Lebesgue and modulation spaces
 }

\author{Divyang G. Bhimani}
\address{Divyang G. Bhimani\\Department of Mathematics\\Indian Institute of Science Education and Research\\ Pune 411008\\India}
\email{divyang.bhimani@iiserpune.ac.in}

\author{Diksha Dhingra}
\address{Diksha Dhingra \\Statmath Unit\\ Indian Statistical Institute\\ Bangalore 560059\\ India}
\email{dikshadd1996@gmail.com}

\author{Vijay Kumar Sohani}
\address{Vijay Kumar Sohani\\ Department of Mathematics\\ Indian Institute of Technology\\ Indore 452020\\ India}
\email{vsohani@iiti.ac.in}
\fancypagestyle{plain}{\fancyhf{} }
\begin{document}
\date{}
\maketitle{}
\begin{center}
DIVYANG G. BHIMANI, DIKSHA DHINGRA,  VIJAY KUMAR SOHANI
\end{center}
\begin{abstract}
The purpose of this work is to provide a broader framework for analyzing the 1D mass-subcritical inhomogeneous nonlinear Schr\"odinger equation (INLS) 
\[iu_t + u_{xx} \pm |x|^{-b}|u|^{\alpha-1}u=0, \quad 1<\alpha< 5-2b\; \text{and}\; 0< b\leq 1/4.\] 
in Sobolev spaces of negative order, reaching almost critical Sobolev spaces by working in Fourier-Lebesgue and modulation  spaces.
More precisely, we establish low-regularity well-posedness
in the Fourier-Lebesgue
$\widehat{L}^{p}$  spaces for $4/3<p<8$. Primarily, we prove global well-posedness for arbitrarily large initial data via the  data decomposition method adapted for the Fourier-Lebesgue spaces, with generalized Strichartz estimates for $p<2$ and classical Strichartz estimates for $p>2$. 
Furthermore, we obtain analogous results for (INLS) in modulation spaces $M^{p,p'}$ for $4/3<p<2$. 


\end{abstract}
\footnotetext{\begin{mathclass} Primary 35Q55, 35Q60, 35R11, 42B37; Secondary 35A01.
\end{mathclass}
}
\tableofcontents
\section{Introduction}
\subsection{Background and motivation}
\par{We investigate the Cauchy problem for the 1D inhomogeneous nonlinear Schrödinger equation (INLS for short), namely,
\begin{equation}\tag{INLS}\label{INLS}
\begin{cases}
    iu_t + u_{xx}+\mu |x|^{-b}(|u|^{\alpha-1}u)=0\\ u|_{t=0}=u_0
\end{cases}(x , t ) \in \mathbb R \times \mathbb R, 
\end{equation}
where $u(x,t) \in \mathbb C,  \mu = \pm 1,\; \alpha>1,  \; b>0 $. This model appears in the setting of nonlinear optics, where the factor $|x|^{-b}$ represents some inhomogeneity in the medium \cite{gill2000optical, liu1994laser}. In recent years, many authors have studied this model; see \cite{guzman2020,genoud2008schrodinger,genoud2010bifurcation,genoud2012inhomogeneous,an2021local,an2021small, JMPAMS, Farahguzman, JMNA, GuzmanMurphy, JMSIAM, LuizMathZ, dinh2021long, Dinhradial, Merle, CarlesINLS}. The  case  $b=0$ corresponds to the standard nonlinear Schr\"{o}dinger equation 
\begin{equation}\tag{NLS}\label{NLS}
iu_t + u_{xx}+\mu |u|^{\alpha-1}u=0,
\end{equation}
it has been extensively
studied over the last three decades \cite{TaoBook, linares2020, cazenave2003semilinear, wang2011harmonic}. 
Recall that \eqref{INLS} is invariant under
scaling \begin{equation}\label{scaling}
u\mapsto \lambda^{\frac{2-b}{\alpha-1}} u (  \lambda x,\lambda^2 t)
\end{equation} and the critical  homogeneous Sobolev spaces $\dot{H}^{s}$ 
index is given by
\[s=s_{b}=\frac{1}{2}- \frac{2-b}{ \alpha-1}, \ \   i.e. \  \alpha= 1 + \frac{4-2b}{1-2s}. \]
We say \eqref{INLS} is mass critical or $L^{2}-$critical  if $s_{b}=0$, i.e.  $\alpha=5-2b$;  and  mass-subcritical if
$0 >s_{b}$, i.e. $1<\alpha< 5-2b$.

Guzm\'an \cite[Theorem 1.9]{guzman2020} proved global well-posedness (GWP for short) for mass-subcritical \eqref{INLS} in  $L^2.$  In fact, Guzmán also proved some local well-posedness (LWP for short)  and small data GWP in  $H^s$  for a specific range of $s\in (0,1)$ and $b$. Later, An and Kim \cite{an2021small} extended this range. Recently, 
Campos,  Correia, and Farah \cite{Campos2025INLS} have refined these results and established the following theorem.  

\begin{theoremA}[well-posedness for inhomogeneous NLS in $H^s(\mathbb R)$, \cite{Campos2025INLS}]\label{WPHs} Let $s\geq 0 $ and  $0<b< 1-s$. 
    \begin{enumerate}
        \item (subcritical) If 
        \begin{align*}
            1< \alpha < \begin{cases}
            1+ \frac{4-2b}{1-2s}\quad &\text{if} \quad s<1/2 \\ 
            \infty \quad &\text{if} \quad s \geq 1/2
        \end{cases},
         \end{align*} 
         then \eqref{INLS} is locally well-posed in $H^s.$ 
        \item (critical) If $s> 1/2$ and $\alpha=1+\frac{4-2b}{1-2s},$ then \eqref{INLS} is locally well-posed in $H^s$ and globally well-posed for the small data.
        \item (mildly ill-posed) If $s>0,\quad 0<b<1 \leq b+s<5/2, \; \alpha-1 \in 2\N \quad \text{and,\; if}\quad b+s \geq 1,\quad b+s \not \in \N,$ then  \eqref{INLS} is mildly ill-posed in $H^s.$ 
    \end{enumerate}    
\end{theoremA}
\subsubsection{\underline{Fourier-Lebesgue spaces}}
The scaling heuristic suggests that we might expect well-posedness for \eqref{INLS} in  $H^s$ for $s\geq s_b$, as depicted in the figure:
\begin{center}
\begin{tikzpicture}
    \draw[<->, very thick, axiscolor] (-6,0) -- (4,0) node[right] {$s$};

    \draw[very thick, axiscolor] (-2,0.12) -- (-2,-0.12)
        node[below=4pt, axiscolor] {$s_{b}=\frac{1}{2}-\frac{2-b}{\alpha-1}$};

    \draw[very thick, axiscolor] (0,0.12) -- (0,-0.12)
        node[below=4pt, axiscolor] {$0$};

    \draw[very thick, axiscolor] (2,0.12) -- (2,-0.12)
        node[below=4pt, axiscolor] {$-\frac{1}{2}+\frac{2-b}{\alpha-1}$};

    \draw[very thick, wellposed] (-2,0) -- (4,0);

    \draw[very thick, illposed] (-6,0) -- (-2,0);

    \draw[
        decorate,
        decoration={brace,amplitude=5pt},
        very thick,
        wellposed
    ]
    (-2,0.3) -- (3.5,0.3)
    node[midway,above=7pt,wellposed] {\bfseries Well-Posed};

    \draw[
        decorate,
        decoration={brace,amplitude=5pt},
        very thick,
        illposed
    ]
    (-5.5,0.3) -- (-2,0.3)
    node[midway,above=7pt,illposed] {\bfseries Ill-Posed};

\end{tikzpicture}
\end{center}
However, all the results discussed so far are restricted to the range $s \geq 0$, leaving the well-posedness regime $s\in (1/2-(2-b)/(\alpha-1), 0)$ unexplored \footnote{\eqref{INLS} does not have Galilean invariance. Hence, there is no critical threshold involved due to Galilean invariance as in the case of classical NLS.}. 
On the other hand,  for $b=0, \alpha=3,$ we have $s_b=-1/2$.  In view of this scaling  and  to fill the 1/2 gap of the derivative,  Gr\"unrock in his seminal work \cite{Grunrock} established  local well-posedness for cubic \eqref{NLS} in \textbf{Fourier-Lebesgue spaces}: 
$$\widehat{L}^p=\widehat{L}^p(\mathbb R)=\left\{f \in \mathcal{S'}(\mathbb R): \|f\|_{\widehat{L}^p} = \|\widehat{f}\|_{L^{p'}}< \infty\right\},$$ 
 where $1\leq p \leq \infty $ and $\frac{1}{p}+\frac{1}{p'}=1$.  
\par{By the Hausdorff–Young inequality, the $\widehat{L}^p$ space satisfies
 \begin{align}\label{embdd1}
 \begin{cases}
     L^p &\hookrightarrow \widehat{L}^p \quad (1\leq p \leq 2), \\
 \widehat{L}^p &\hookrightarrow L^p \quad (2\leq p \leq \infty).
\end{cases}
 \end{align}
Though $\widehat{L}^p$ spaces are analogous to $L^p$ spaces, $\widehat{L}^p$ spaces  offer a more appropriate setting for initial data in studying \eqref{NLS}. The Schr\"odinger propagator $e^{it\partial_x^2}:L^p\to L^p $ iff $p=2$. While the Schr\"odinger propagator is unitary on $\widehat{L}^p$ for all $p,$ in fact, we have
\begin{equation}\label{spfls}
    \|e^{i t \partial_x^2}f\|_{\widehat{L}^p}=\|f\|_{\widehat{L}^p} \; \forall t\in \R.
\end{equation}}
\par{ The scaling \eqref{scaling} leaves homogeneous Fourier-Sobolev space $\|f\|_{\widehat{L}^p_s}=\||\cdot|^s\hat{f}\|_{L^{p'}}-$ norm invariant for 
\begin{equation}\label{flsscaling}
s=s_{\text{crit}}(p)=\left(1-\frac{1}{p'}\right)-\frac{2-b}{\alpha-1}.
\end{equation}
In particular, when $s=b=0, \alpha=3,$ cubic  \eqref{NLS} is scaling critical in $\widehat{L}^1$. Thus, it is expected from these scaling considerations that the Cauchy problem for the cubic \eqref{NLS} is locally well-posed in $\widehat{L}^p$ for $1 < p \leq \infty $.}
 It is known that the cubic \eqref{NLS} with a Dirac delta function as initial data is ill-posed. Gr\"unrock \cite{Grunrock} established LWP for cubic \eqref{NLS}  in $\widehat{L}^p \ (1<p< \infty)$ almost reaching critical index $p=1$. See also \cite[Remarks 1.2 and 1.3]{guo20171d}. We also point out that a precursor to this result by Vargas-Vega \cite{vargas2001global}, establishing global well-posedness for data with an infinite $L^{2}-$norm.  Later, Hyakuna and Tsutsumi \cite{hyakuna2012existence}  strengthened Gr\"unrock's result in $\widehat{L}^{p}$ to all mass-subcritical cases $\alpha \in (1,5)$. In the following theorem, we summarize the known well-posedness results in $\widehat{L}^p$.}
\begin{theoremA}[well-posedness in $\widehat{L}^p$]\label{WPFLS}\
\begin{enumerate}
\item Gr\"unrock \cite[Section 5]{Grunrock} proved \eqref{NLS} is:
\begin{itemize}
    \item[--] LWP for $\alpha=3$ and $1<p<\infty$. 
    \item[--] GWP  for $\alpha=3$ and $5/3<p<2$. 
\end{itemize}    
    \item Hyakuna-Tsutsumi \cite[Theorems 1 and 2]{hyakuna2012existence} proved \eqref{NLS} is:
    \begin{itemize}
     \item[--] LWP for $\alpha \in (1,5)$, $4/3<p<2$ or $2<p<\infty.$ 
    \item[--] GWP for $\alpha \in (1,5)$, $4/3<p<2$ or $2<p<\infty,$ with $p$ sufficiently close to 2.    
\end{itemize}
\item Bhimani-Carles \cite[Theorems 1.2 and 1.3]{BhimaniNorm} proved \eqref{NLS} is
\begin{itemize}
    \item[--] strong ill-posedness below the expected critical regularities ($s<s_{c}$).
\item[--] norm inflation with infinite loss of regularity. 
\end{itemize}
\end{enumerate}    
\end{theoremA}
Taking Theorem \ref{WPHs} into account, all the existence results for \eqref{INLS} are restricted to Sobolev spaces with nonnegative index, i.e., in $H^{s}(\R),\; s \geq0.$ In view of this, together with the discussion in Theorem  \ref{WPFLS},   we are inspired to study \eqref{INLS} in the $\widehat{L}^p$ spaces for $1<p< \infty$, where the analysis differs according to whether $p<2$ or $p>2$. See Theorems \ref{lwp2}-\ref{gwp3} below. 

\subsubsection{\underline{Modulation spaces}}
We shall now turn our attention to the modulation spaces \cite{Feih83, wang2011harmonic, KassoBook}. Let $\rho: \mathbb R \to [0,1]$  be  a smooth function satisfying   $\rho(\xi)= 1 \  \text{if} \ \ |\xi| \leq 1/2$ and $\rho(\xi)=
0 \  \text{if} \ \ |\xi| \geq  1$. Let  $\rho_k$ be a translation of $\rho,$ that is,
$ \rho_k(\xi)= \rho(\xi -k) \ (k \in \mathbb Z).$
Denote 
$$\sigma_{k}(\xi)= \frac{\rho_{k}(\xi)}{\sum_{l\in\mathbb Z}\rho_{l}(\xi)}\quad (k \in \mathbb Z).$$
The frequency-uniform decomposition operators can be  defined by 
$$\square_k = \mathcal{F}^{-1} \sigma_k \mathcal{F} \quad (k \in \mathbb Z)$$ 
where $\mathcal{F}$ and $\mathcal{F}^{-1}$ denote the Fourier and inverse Fourier transform respectively. The \textbf{weighted modulation spaces } $M^{p,q}_s \ (1 \leq p,q \leq \infty, s \in \R)$ is defined as follows:
\begin{equation*}
    M^{p,q}_s= M^{p,q}_s(\R)= \left\{ f \in \mathcal{S}'(\R): \left|\left|f\right|\right|_{M^{p,q}_s}=  \left\| \|\square_kf\|_{L^p_x} (1+|k|)^{s} \right\|_{\ell^q_k}< \infty  \right\} . 
\end{equation*}
For $s=0,$ we write $M^{p,q}_0= M^{p,q}.$ For $p=q=2,$ modulation spaces coincide with Sobolev spaces, i.e. $M^{2,2}_s= H^s \ (s \in \R). $ In the last two decades, many authors have studied \eqref{NLS}  with Cauchy data in $M^{p,q}$-spaces; we briefly summarize some of them in the theorem below and refer to \cite{KassoBook, wang2011harmonic} for a thorough introduction.
\begin{theoremA}[well-posedness in $M^{p,q}$ ]\label{WPMS}
 Let $1\leq p, q \leq \infty$ and $s\geq 0.$ We have following results in \eqref{NLS}:
\begin{enumerate} 
\item  \cite{Wang2006, wango, Kasso2009, Bhimani2016, CorderoJMAA} LWP in $M_s^{p,1} $ for $  \alpha-1 \in 2\mathbb N$ and small-data GWP  in $M^{2,1}$  for  $\alpha=3$. 
\item GWP in 
\begin{itemize}
    \item[--]  \cite{oh2018global, guo20171d} $M^{2,q}$ for $2\leq q< \infty$ and $\alpha=3$. 
    \item[--] \cite{Klaus} $M^{p,q}$ for $1\leq q\leq  p \leq 2$ or $M^{p,q}_s$ for $s>1/q', 1\leq p \leq 2$  or $M_1^{p,1}$ for $2\leq p< \infty.$
    \item[--] \cite{LeonidIn,leonidthesis, Bhimani20251d}   $M^{p, p'}$ for $p$ sufficiently close to 2, and $\alpha \in (1,5)$.
\end{itemize}
\end{enumerate}
\end{theoremA}
It is known that modulation spaces do not possess a strict scaling-invariance property; however, they scale like   Fourier-Lebesgue spaces; see \cite[Section 1.3]{BhimaniNorm}. Further, taking the following 
 sharp embedding
 \begin{align}\label{embdModFLS}
    M^{p,p'} &\hookrightarrow \widehat{L}^{p} \quad \text{for} \ 
     \;  1 \leq p \leq 2
\end{align}
and uniform boundedness of Schr\"odinger propagator modulation spaces
\begin{equation}\label{Mpp1}
    \|e^{it\Delta}f\|_{M^{p,q}} \lesssim C(t) \|f\|_{M^{p,q}} \quad \text{for} \quad  1 \leq p,q \leq \infty
\end{equation}
into account,  we expect  well-posedness of \eqref{INLS} in the space $M^{p,p'}$ for $p<2$. See Corollaries \ref{lwp} and \ref{mr} below. We also refer to \cite{Bhimani2024low} for the well-posedness theory for \eqref{INLS} in $M^{p,p'}$ for $2<p$.

\subsection{Statement of the main results}
\par{ Let us  fix any potential index  $b\in (0,1/4]$. Then for a given nonlinearity index  $\alpha \in (1,5-2b)$, we define solution space index $r(\alpha)$ \footnote{Note that $r$ is a well-defined and continuous function of $\alpha$ and $r>2$ in all cases.}  as follows
\begin{equation}\label{rvalues}r=r(\alpha):= 
\begin{cases}
 \alpha+1 &\text{if} \quad \alpha \in  (1,3) \\ 
 2(\alpha-1)\quad  & \text{if} \quad \alpha\in [3,5-2b).\\ 
\end{cases}
\end{equation}}

\par{
We begin with stating the results for the well-posedness theory $\widehat{L}^{p}$ for $p<2$ and then treat the case $p>2$ in the subsequent subsection.}

\subsubsection{\underline{Well-posedness in $\widehat{L}^{p}$ and $M^{p,p'}$ for $p<2$}}  
\par{To state our results, we need the following definitions and notations.}
\begin{Definition}[Generalized Strichartz pair, \cite{hyakuna2012existence, CFefferman,grunrock2004improved}] \label{gsp} Let $1\leq p_{0} \leq 2.$
 Consider a pair of exponents $(Q_{p_{0}}(r),r)$ satisfying 
\begin{align*}
&\frac{2}{Q_{p_{0}}(r)} + \frac{1}{r} = \frac{1}{p_{0}} \\
  \text{either}\quad 
& 0 < \frac{1}{Q_{p_{0}}(r)} < \min \left( \frac{1}{2} - \frac{1}{r},\; \frac{1}{4} \right) \quad  \text{and}\quad
  0 < \frac{1}{r} < \frac{1}{2}\\
&  \text{or}\quad 
(Q_{p_0}(r), r)= (4, r) \quad  \text{and}\quad r > 4.
\end{align*}
The set of all such pairs is denoted by 
$\widehat{\mathcal{X}}(p_{0}).$ 
\end{Definition}
\begin{Remark}
    With $r$ given in \eqref{rvalues}, $(q(r),r) \in \mathcal{A}$ (as defined in \eqref{aps}.) and $(Q_{p_{0}}(r),r) \in \mathcal{X}(p_{0})$ for $p_{0}$ satisfying conditions as listed in Remark \ref{whyp0so}.
    \end{Remark}
The possible values of $1/p_{0}$ and $1/r$ qualifying for $(Q_{p_{0}}(r), r))\in \widehat{\mathcal{X}}(p_{0})$ are illustrated by Figure \ref{fig2:myplot}:
\begin{figure}[H]
\centering
\begin{tikzpicture}[scale=6.7]
\draw[->] (0,0) -- (1,0) node[right] {$\frac{1}{p_{0}}$};
\draw[->] (0,0) -- (0,0.7) node[above] {$\frac{1}{r}$};

\foreach \x/\label in {0/0, 0.5/{\frac{1}{2}}, 0.75/{\frac{3}{4}}}
    \draw (\x,0) -- (\x,-0.02) node[below] {$\label$};
\foreach \y/\label in {0.25/{\frac{1}{4}}, 0.5/{\frac{1}{2}}}
    \draw (0,\y) -- (-0.02,\y) node[left] {$\label$};
\fill[cyan!40]
(0.5,0.25) --
(0.75,0.25) --
(0.5,0.5) -- cycle;

\node at (0.583,0.333) {$R_1$};
\fill[green!40]
(0.5,0) --
(0.5,0.25) --
(0.75,0.25) -- cycle;
\node at (0.583,0.167) {$R_2$};
\draw[dotted, thick] (0.5,0.5) -- (0.75,0.25);
\draw (0.5,0.25) -- (0.75,0.25);
\draw (0.5,0) -- (0.5,0.5);
\draw[red, thick] (0.5,0) -- (0.75,0.25);
\draw[fill=white] (0.5,0.5) circle (0.012);
\draw[fill=white] (0.75,0.25) circle (0.012);
\fill (0.5,0) circle (0.012);
\fill (0.5,0.25) circle (0.012);
\end{tikzpicture}
 \caption{The cyan region $R_{1}$ is defined as the area  bounded by $1/4\leq 1/r< 1/2,\; 1/2 \leq 1/p_{0} < 3/4 \;  \text{and} \; 1/{p_{0}}+1/r<1.$ The green region $R_{2}$ represents the area bounded by $0 < 1/r \leq 1/4,\; 1/2 \leq 1/{p_{0}}< 3/4 \; \text{and}\; 1/{p_{0}}-1/r< 1/2$, while the red line denotes $0 \leq 1/r < 1/4,\;1/{p_{0}}-1/r= 1/2$. }
 \label{fig2:myplot}
\end{figure}
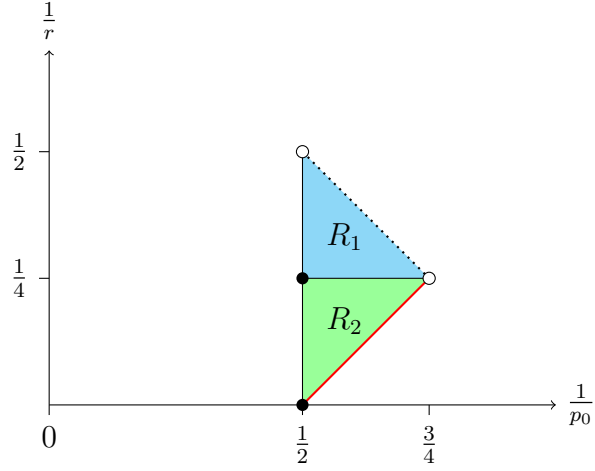
In all subsequent theorems, we assume that 
$p_{0}$ satisfies  

 \begin{gather}\label{pminloc}
2>p_{0}>\begin{cases}
\vspace{0.2cm}
\max \left( \frac{\alpha+1}{\alpha}, \frac{4\alpha}{9-4b} \right) &\text{if} \quad \alpha \in  (1,3) \\ 
\vspace{0.2cm}
\max \left( \frac{2(\alpha-1)}{\alpha}, \frac{2\alpha}{5-2b} \right) \quad  & \text{if} \quad \alpha\in [3,5-2b).
\end{cases}
    \end{gather}

  \begin{thm}[local well-posedness]\label{lwp2} Let $\alpha \in (1,5-2b)$,  $b\in (0, 1/4]$ and $r$ be given as in \eqref{rvalues}. Assume that  $u_{0}\in \widehat{L}^{p_{0}}$ with $p_{0}$ satisfying \eqref{pminloc} \footnote{We refer to Remark \ref{Remarkwhyp0so} for a detailed rationale for the selected range of $p_{0}$ in Theorem \ref{lwp2}.}.
    Then, for $ (Q_{p_{0}}(r),r)\in \widehat{\mathcal{X}}(p_{0})$,
$\eqref{INLS}$  has a mild solution up to time $T^{*}=T^*(\|u_{0}\|_{\widehat{L}^{p_{0}}},\alpha,b)>0$ satisfying the corresponding integral equation \eqref{DP} in
$L^{Q_{p_{0}}(r)}([0,T'),L^{r}(\R))$
for any $T'\in (0,T^{*})$.
     \end{thm}
\begin{Remark}\label{Remark:1} We discuss several comments on Theorem \ref{lwp2} below.
\begin{enumerate}
\item The condition on $b$ in all of our theorems arises from the bound for the inhomogeneous nonlinearity and, in particular, by the requirement of choosing a suitable admissible pair satisfying \eqref{aps}; see also Remark \ref{Remarkwhyp0so}\eqref{whybso}. On the other hand, the condition on $\alpha$ is imposed because we consider the mass-subcritical regime. Moreover, our global analysis relies on the global well-posedness theory for the mass-subcritical equation \eqref{INLS} in $L^{2}$-spaces; refer to  \cite[Theorem 1.8]{guzman2020}.
\item From the scaling consideration \eqref{flsscaling}, it is expected that the Cauchy problem for the cubic \eqref{INLS} is locally well posed in $\widehat{L}^p$ for $2/(2-b) < p \leq \infty $. Theorems \ref{lwp2}-\ref{gwp3} covers the well-posedness theory in $\widehat{L}^p$ for \footnote{Note that the  critical index $p_{c}=2/(2-b)<4/3$ for $0<b<1/2$.}  $4/3 < p < 8 $ .
\item \label{difference1} The space $\widehat{L}^{p}$ scales like $\dot{H}^{s}$ for $s=\frac{1}{2}-\frac{1}{p}$. Thus, Theorems \ref{lwp2}-\ref{mr2} establish local and global well-posedness theory for cubic \eqref{INLS}  in $\widehat{L}^p \ (4/3<p<2)$ almost reaching critical index $p=2/(2-b) $. Notice that when $p \to 2/(2-b) ^{+}$, we
have $s \to (-1/2+b/2) ^{+}$, and 
$\dot{H}^{\frac{-1}{2}+\frac{b}{2}}$ and $\widehat{L}^{\frac{2}{2-b}}$ are both the scaling invariant spaces. We can say
in this sense that the scaling invariant index is almost reached.
\item  Moreover, by the Sobolev embedding (in the Fourier side) \cite[Lemma B.1.]{MasakiA&PDEs} yields
 \begin{align}\label{embdd2}
     \dot{H}^{\frac{1}{p}-\frac{1}{2}} &\hookrightarrow \widehat{L}^p \quad (1< p \leq 2).
 \end{align} 
 Thus, the well-posedness in $\widehat{L}^{p}$ for $4/3<p<2$ (established in Theorems \ref{lwp2}-\ref{mr2}) covers the well posedness in the $H^{s}$-spaces between $L^{2}$ and $H^{1/4}$, complementing Theorem \ref{WPHs} by extending the existence of solution globally in Theorem \ref{mr2}.
\item In Theorem \ref{lwp}, we can cover initial data (highly oscillatory and rough data) that has not been addressed in the literature up to now. For instance, define a function on the Fourier transform side by
\[\widehat{f}(\xi)= \frac{1}{|\xi|^{\sigma_1}} \chi_{|\xi|\geq 1}\]
where $\sigma_1 p' >1$ so that $f\in \widehat{L}^{p}$ for $p\in (4/3,2)$.
We may choose some $\sigma_1>1/3$ so that $f\in \widehat{L}^{3/2}$ but $f\notin H^{s}(\mathbb R)$ for $s=-1/6+\epsilon$ and some $\epsilon>0$ (not covered in Theorem \ref{WPHs}) since
\[\int_1^{\infty} (1+ \xi^2)^{s} |\hat{f}(\xi)|^2 d\xi \simeq \int_1^{\infty} \xi^{2(s- \sigma_1)} d\xi = \infty.\]
for $s-\sigma_1 \geq -1/2.$  \\
The key point is that Theorem \ref{lwp} could cover initial data whose  Fourier transform decays very slowly (and so the  original data could be very rough). Specifically, the decay is slow, and so it is not in $H^s$; however,  the decay is sufficient to be in the Fourier Lebesgue spaces  $\widehat{L^p}$. Thus, $\widehat{L^p}$ seems to be the right choice of initial data to treat the low-regularity initial data.
\end{enumerate}
\end{Remark}
\begin{Remark}We briefly discuss the proof strategy for Theorem \ref{lwp2}.
\begin{enumerate}
     \item[-] To bound the homogeneous part of \eqref{INLS}, we make use of the generalized Strichartz estimates for the Schrödinger equation in Fourier-Lebesgue spaces, widely known as the Fefferman–Stein estimate, which traces back to \cite{CFefferman}. This estimate can be expressed as
\begin{equation}\label{gse1}
    \|e^{i t \partial_x^2}f\|_{L^{Q_{p_{0}}(r)}(\R,L^{r})} \lesssim \|f\|_{\widehat{L}^{p_{0}}}
\end{equation}    
for $4/3< p_{0}\leq 2$ and any $(Q_{p_{0}}(r),r) \in \widehat{\mathcal{X}}(p_{0})$.
Also see \cite[Theorem 5, Lemma 4]{hyakuna2012existence},\cite[Proposition 2.4]{masaki2016}.
\item[-] To bound the inhomogeneous part of \eqref{INLS}, the main difficulty is to bound the spatially decaying factor $|\cdot|^{-b}$ in the nonlinearity, as $|\cdot|^{-b}$ does not belong to any Lebesgue spaces. 
\item[-] We handle this by decomposing $\mathbb R =B\cup B^{c}$ into two parts, here $ B = B(0,1)$  denotes the unit ball centred at the origin in $\R$ and $B^{c}=\R \setminus B.$  Specifically, 
\begin{equation*}
\begin{cases}
|\cdot|^{-b}\in L^{\gamma}(B) & if  \  1/\gamma - b > 0\\
|\cdot|^{-b}\in L^{\gamma}(B^c) & if  \  1/\gamma - b < 0.
\end{cases}
\end{equation*}
\item[-] We need to select different admissible pairs for the regions inside the ball, i.e. in $B$, and outside the ball, i.e. in $B^c$, which is crucial for handling
nonlinear interactions, refer to Lemma \ref{lemlwp1}. These admissible pairs play a quite important role in determining the range of $b$, see Remark \ref{Remarkwhyp0so}\eqref{whybso}. 
\item[-] Subsequently, by  employing a standard fixed-point argument in a closed subset of the solution 
space $L^{Q_{p_{0}}(r)}_{T}L^{r}(\R)$, we establish the local well-posedness in Theorem \ref{lwp2}. 
\end{enumerate}
\end{Remark}
\par{Also, we obtain analogous results for \eqref{INLS} in $M^{p,p'}$ for $4/3<p<2$ in view of \eqref{embdModFLS}. }

\begin{cor}[local well-posedness]\label{lwp}Let $\alpha \in (1,5-2b)$,  $b\in (0, 1/4] $ ,  $r$ and $p_{0}$ be as given in \eqref{rvalues} and \eqref{pminloc}, respectively. Assume that  $u_{0}\in M^{p_{0},p_{0}'}$. Then, for $ (Q_{p_{0}}(r),r)\in \widehat{\mathcal{X}}(p_{0})$, $\eqref{INLS}$  has a mild solution up to time $T^{*}=T^*(\|u_{0}\|_{ M^{p_{0},p_{0}'}},\alpha,b)>0$ satisfying the corresponding integral equation \eqref{DP} in
$L^{Q_{p_{0}}(r)}([0,T'),L^{r}(\R))$
for any $T'\in (0,T^{*})$.
\end{cor}

The local solution established in Theorem \ref{lwp2} and Corollary \ref{lwp} can be extended to a global one \footnote{In this case, $T^{*}=\infty$ in Theorem \ref{lwp2}; the Cauchy problem for \eqref{INLS} is said to be \emph{globally well-posed} (GWP for short).} under certain restrictions on the exponent $p$ of the $\widehat{L}^{p}-$spaces. Before stating the global well-posedness results, we introduce the following notation.
\par{We denote}
\begin{eqnarray*}
\footnotesize{\frac{1}{\nu}= 
    \begin{cases}
       \displaystyle \left(\sup_{\alpha \in (1,3)}  \frac{\alpha}{2r}\right)-\frac{\alpha}{2r}+\frac{b}{2}\; &\text{if} \quad \alpha \in  (1,3)\\
    \displaystyle \left(\inf_{\alpha \in [3,5-2b)}  \frac{\alpha}{2r}\right)-\frac{\alpha}{2r}+\frac{b}{2} &\text{if} \quad \alpha\in [3,5-2b)\\
    \end{cases}
        =
    \begin{cases}\vspace{0.2cm}
       \displaystyle \frac{3}{8}-\frac{\alpha}{2r}+\frac{b}{2}\; &\text{if} \quad \alpha \in  (1,3)\\
        \displaystyle  \frac{5-2b}{8(2-b)}-\frac{\alpha}{2r}+\frac{b}{2} &\text{if} \quad \alpha\in [3,5-2b).\label{nu}
    \end{cases}}
         \end{eqnarray*}
and
\begin{equation}\label{complexpmin}p_{min}:=
    \begin{cases}
      \frac{(\alpha-1)(3p_{0}+2)}
{4 +2(\alpha-2)p_{0}-\frac{2}{\nu}(2-p_{0})} \quad &\text{if} \  \alpha-2 +\frac{\alpha-1}{2p_{0}}+\frac{1}{\nu}>0\\
p_{0} \quad  & \text{otherwise}
    \end{cases}
\end{equation}

  \begin{thm}[global well-posedness]\label{mr2} Let $\alpha \in (1, 5-2b)$, $b\in (0, 1/4], \; r$ and $p_{0}$ be as given in \eqref{rvalues} and \eqref{pminloc}, respectively. Assume that  for $p \in (p_{0},2]$, $u_{0}\in \widehat{L}^{p}$  and $p$ be such that $p_{\min}<p\leq 2.$ Then, $\eqref{INLS}$  has a unique global solution 
    \begin{equation*}
  u \in  \left(C(\R,L^2)\cap L^{q(r)}_{loc}(\R,L^{r}) \right)+\left( C(\R,\widehat{L}^{p_{0}})\cap L^{Q_{p_{0}}(r)}_{loc}(\R,L^{r})\right).
\end{equation*} 
     \end{thm}

\begin{Remark}
\begin{enumerate}
\item Our method of proof for proving Theorem \ref{mr2} is based on data decomposition method, introduced by Bourgain \cite{Bourgain1999} to establish GWP for NLS having power-type nonlinearity in $H^{s}$ for $s>3/5$. Later, Vargas and Vega \cite{vargas2001global} adapted this method to handle the infinite-mass case.
Also, see \cite[Section 3.2]{KenigonBourgain}, \cite{bhimani2024fractional, LeonidIn, leonidthesis} and \cite[Section 3.9]{TaoBook} for more details.
\item Hyakuna et al. \cite{hyakuna2012existence} have  successfully adapted this method  for classical \eqref{NLS} (i.e. $b=0$) in $\widehat{L}^{p}$ spaces. In this paper, we adapted this method to study \eqref{INLS}. We note that due to the singular potential $|\cdot|^{-b}$ with $b>0$, the problem becomes significantly more challenging and leads to a loss of regularity in the solution space, so that the solution does not, in general, remain in the same space as the initial data. This loss is attributed to the inability to control the singular potential solely through Lebesgue-space estimates.
\end{enumerate}
\end{Remark}

    \begin{cor}[global well-posedness in $M^{p,p'}$]\label{mr}Let  $\alpha, r, p_0$ and $p$ be as in Theorem \ref{mr2}.
     Assume that  $u_{0}\in M^{p,p'}(\mathbb R)$. Then, $\eqref{INLS}$  has a unique global solution 
    \begin{equation*}
  u \in  \left(C(\R,L^2)\cap L^{q(r)}_{loc}(\R,L^{r})\right) + \left(C(\R,M^{p_{0},p_{0}'})\cap L^{Q_{p_{0}}(r)}_{loc}(\R,L^{r})\right).
\end{equation*} 
     \end{cor}
\begin{Remark}
We also note that modulation spaces have the following inclusion relations:
    \begin{equation}\label{srp1}  M_{s_1}^{p_{1},q_{1}}\hookrightarrow M_{s_2}^{p_{2},q_{2}} \quad \text{whenever} \quad
    \begin{cases}
        p_{1}\leq p_{2}, q_{1}\leq q_{2},  s_2 \leq s_1,\\
        p_{1}= p_{2}, q_{2}<q_{1},  s_1-s_2 > \frac{1}{q_2}-\frac{1}{q_1}, 
          \end{cases}
       \end{equation}
for $p_{i},q_{i} \in [1,\infty], (i=1,2)$ and $s_1, s_2 \in \R$, also refer to \cite{KassoBook,wang2011harmonic}.
     Taking into account \eqref{srp1}, we have
\begin{eqnarray}\label{dsei}
   M^{p,p'}(\mathbb R) \hookrightarrow  H^{s}(\mathbb R) 
\quad \text{for} \quad  s < \frac{1}{2} - \frac{1}{p},  \ \  p<2,
\end{eqnarray}
 and there exist $u_0\in M^{p,p'} (\mathbb R)   (1\leq p<2)$ such that $u_0\notin L^2(\mathbb R)$. Thus we have $$u_0\in (H^s \setminus L^2)(\mathbb R) \cap M^{p,p'} (\mathbb R). $$  
 Surprisingly, Corollaries \ref{lwp} - \ref{mr} reveals that there is data in $(H^s\setminus L^2) \cap M^{p,p'}$ for which \eqref{INLS} is well-posed.
 \end{Remark}

\subsubsection{\underline{Well-posedness in $\widehat{L}^{p}$ for $2<p$}}
\begin{thm}[local well-posedness]\label{lwp3} Let $\alpha \in (1, 5-2b), b\in (0, 1/4]$ and $r$ be given as in \eqref{rvalues}.
  Assume that $u_{0}\in L^2+\widehat{L}^{r}.$ Then, there exists $T^*=T^*(\|u_{0}\|_{L^2+\widehat{L}^{r}},\alpha,b)>0$ and a mild solution up to time $T^{*}=T^*(\|u_{0}\|_{L^2+\widehat{L}^{r}},\alpha,b)>0$ satisfying the corresponding integral equation \eqref{DP} in
\begin{equation*}
   \left(L^{\infty}([0,T'),L^2)~\cap~ L^{q(r)}([0,T'),L^{r})\right) ~+~L^{\infty}([0,T'),\widehat{L}^{r})
\end{equation*}
for any $T'\in (0,T^{*})$.
\end{thm}
\par{The local solution established in Theorem \ref{lwp3} can be extended to a global one under certain restrictions on the upper bound of the exponent $p$ for $p \in (2,r)$.}
\par{Before stating the global well-posedness results, we denote \begin{equation}\label{pmax}p_{\max}\footnote{See Remark \ref{detailedpmax} for the precise expression for $p_{\max}$.} := 
\begin{cases}
\frac{(\alpha-1)(\frac{3}{2}+\frac{1}{r})}
{\frac{\alpha-1}{2}(\frac{3}{2}+\frac{1}{r})+(\frac{5-\alpha}{2}-\frac{2}{\nu})(\frac{1}{r}-\frac{1}{2})}   \quad &\text{if} \  (\alpha-1)(\frac{3}{4}+\frac{1}{2r})-\frac{5-\alpha}{4}+\frac{1}{\nu}>0 \\ 
r \quad  & \text{otherwise}.\end{cases}
\end{equation}}

\begin{thm}[global well-posedness]\label{gwp3}  Let $\alpha \in (1, 5-2b), b\in (0, 1/4]$ and $r$ be given as in \eqref{rvalues}. Assume that  $u_{0}\in \widehat{L}^{p}$ for $p\in (2,p_{max}).$ Then, the local well-posedness established in Theorem \ref{lwp3} can be extended globally and \eqref{INLS} has a unique global solution
\begin{equation*}
    u\in \left(C(\R,L^2) \cap L^{q(r)}_{loc}(\R,L^{r})\right)+ C(\R,\widehat{L}^{r}).
\end{equation*}
\end{thm}
\begin{Remark}\label{b02leqp}
\begin{enumerate}
\item Applying Sobolev embedding in the Fourier side \cite[Lemma B.1.]{MasakiA&PDEs} yields
\begin{align}\label{embdd3}
 \widehat{L}^p &\hookrightarrow  \dot{H}^{\frac{1}{p}-\frac{1}{2}} \quad (2\leq p < \infty).
 \end{align}
Thus, there exist $$u_0\in (H^s \setminus L^2) \cap \widehat{L}^p$$ for some $s<0$ and $2<p$ in which \eqref{INLS}  experiences well-posed in Theorems \ref{lwp3}-\ref{gwp3}.
\item It is interesting to note that Theorems \ref{lwp3}-\ref{gwp3} are applicable to initial data from Sobolev spaces  in view of the following sharp embedding; see \cite[Lemma A.51.]{leonidthesis} when $p\in (2, \infty]$ and $s>1/2-1/p$:
\begin{align}\label{embdd4}
H^{s}(\R)  \hookrightarrow \widehat{L}^p (\R).
 \end{align}
\item Theorems \ref{lwp3}-\ref{gwp3} generalize the well-posedness results established in \cite{Bhimani2024low}, where \eqref{INLS} was studied in the modulation spaces \(M^{p,p'}\) for \(p>2\) and \(b\in \left(0,(3-\sqrt{7})/2\right]\). In the present work, we complement the well-posedness for \eqref{INLS} in $\widehat{L}^{p}$ spaces for $2<p$ and also, extend the admissible range of the singularity parameter to \(b\in (0,1/4]\).
Refer to \cite{Bhimani2024low} for the well-posedness theory for \eqref{INLS} in $M^{p,p'}$ for $2<p$.
\item The diﬀerence between the case $p<2$ and $p>2$ is that, in $p<2$, our well-posedness theory covers Sobolev spaces of negative indices, while $p>2$ covers positive-low Sobolev regularities since the space $\widehat{L}^{p}$ scales like $\dot{H}^{\frac{1}{2}-\frac{1}{p}}$; also see Remark \ref{Remark:1} \eqref{difference1}.
\item In contrast to the case $p<2$, the case $p>2$ follows from the standard Strichartz estimates \eqref{st}. Thus, we obtain well-posedness in the usual Strichartz space $L^{q(r)}_{T} L^{r}(\R)$ with $(q(r),r)\in \mathcal{A}$ and some $T>0$. 
\end{enumerate}
\end{Remark}

\par{This article is organized as follows. In Section \ref{np}, we introduce the basic notations and preliminary tools. Section \ref{s3} provides an important lemma which serves as a fundamental step toward estimating the nonlinearity. Section \ref{s4} is devoted to proving Theorem \ref{lwp2}, Corollary \ref{lwp}, Theorem \ref{mr2} and Corollary \ref{mr}. While, Section \ref{s5} contains the proofs of Theorems \ref{lwp3} and \ref{gwp3}.} 

\section{Notations and preliminaries}\label{np}
\par{We begin by introducing some notations and preliminary estimates, which will serve as the framework for the theorem proved in the following section.}
  \subsection{Notations}\noindent 
  \begin{itemize}
      \item[-] The symbol $X \lesssim Y$ means 
 $X \leq CY$ 
for some constant $C>0.$ While $X \approx Y $ means $C^{-1}X\leq Y \leq CX$ for some constant $C>0.$
 \item[-] The norm of the space-time Lebesgue spaces $L^{q}([0,T],L^{r}(\R))$ is defined as
$$\|u\|_{L^{q}_{T}L^{r}}:=\|u\|_{L^{q}([0,T],L^{r}(\R))}=\left(\int_{0}^{T} \|u(\cdot,t)\|_{L^{r}(\R)}^{q} dt\right)^{\frac{1}{q}}.$$
We simply write $\|u\|_{L^{q}L^{r}}$ in place of $\|u\|_{L^{q}(\R,L^{r}(\R))}.$ 
\item[-] For $p\in [1, \infty]$,  we denote $p'$ the H\"older conjugate, i.e.  $\frac{1}{p}+\frac{1}{p'}=1.$ 
\item[-] Denote \begin{equation}\label{expforG} G(u,v,w)=|u+v|^{\alpha-1}(u+v)-|u+w|^{\alpha-1}(u+w)
\end{equation}
for  $\alpha>1$ and $u,v,w\in \C.$  
  \end{itemize}

\subsection{Preliminaries} 
\begin{itemize}
\item[--]  (e.g. Lemma 3.9 in {\cite{leonidthesis}}). 
For $G$ defined in \eqref{expforG} and $\alpha>1$, the following holds:
\begin{equation}\label{eg}
         |G(u,v,w)| \lesssim_{\alpha} (|u|^{\alpha-1}+|v|^{\alpha-1}+|w|^{\alpha-1})|v-w|.
\end{equation}
   
    \item[--] 
Define the \emph{Schr\"odinger propagator} $e^{i t \partial_x^2}$ as follows:
\begin{eqnarray}
\label{sg}
e^{i t \partial_x^2}f(x):= \int_{\R} e^{i\pi t|\xi|^{2}} \widehat{f}(\xi) e^{2\pi i  \xi \cdot x}d \xi \quad (f\in \mathcal{S}, t \in \mathbb R).
\end{eqnarray}
\item[--]  A pair $(q,r)$ is said to be \emph{Schr\"odinger admissible} if  
\begin{equation}\label{aps}
    q\geq 2, \quad r\geq 2, \quad \frac{2}{q} + \frac{1}{r} =   \frac{1}{2}.
\end{equation}    
The set of all such admissible pairs is denoted by $$\mathcal{A}= \{(q,r):(q,r) \; \text{is an admissible pair} \}.$$
\item[--] \emph{(Strichartz estimates for  Schr\"odinger equation, \cite{linares2020,cazenave2003semilinear,KeelTao1998}).} Denote
$$DF(x,t):=  e^{i t \partial_x^2}u_{0}(x)  \pm i\int_0^t  e^{i (t-s) \partial_x^2}F(x,s) ds.$$
Assume $u_{0} \in L^2$ and $F \in L^{q_2'} (I, L^{r_2'}).$  Then for any time interval $I\ni0$ and 2-admissible pairs $(q_j,r_j)$, $j=1,2,$ 
satisfying 
\begin{equation}\label{st} 
    \|D(F)\|_{L^{q_1}(I,L^{r_1})}  \lesssim  \|u_{0} \|_{L^2}+    \|F\|_{L^{q'_2}(I,L^{r'_2})}
\end{equation}    
 where $q_j'$ and $ r_j'$ are H\"older conjugates of $q_j$ and $r_j$
respectively.
\item[--] \emph{(Interpolation of $ \widehat{L}^{p}-$spaces with $p<2$, \cite[Section 4, Theorem 2(i)]{hyakuna2012existence}).} Let $u_{0} \in \widehat{L}^{p},\; p\in (p_{0},2)$ and $N>0$.
 Then, there exist $\phi_{0} \in L^{2}$ and $\psi_{0} \in \widehat{L}^{p_{0}}$  such that $u_{0}= \phi_{0}+ \psi_{0},$
\begin{eqnarray}\label{ipt2}
 \|\psi_{0}\|_{\widehat{L}^{p_{0}}}\leq C \|u_{0}\|
_{\widehat{L}^{p}}\frac{1}{N}, \quad \|\phi_{0}\|_{L^2}\leq C \|u_{0}\|_{\widehat{L}^{p}}N^{\beta_{1}} \quad \text{and} \quad
\beta_{1}=\frac{\frac{1}{p}-\frac{1}{2}}{\frac{1}{p_{0}}-\frac{1}{p}}. 
\end{eqnarray}
\item[--] \emph{(Interpolation of $M^{p,p'}-$spaces with $p<2$, \cite[Lemma 6.1]{Bhimani20251d}).}
Let $u_{0} \in M^{p,p'},\; p\in (p_{0},2)$ and $N>0$. Then, there exist $\phi_{0} \in L^{2}$ and $\psi_{0} \in M^{p_{0},p_{0}'}$  such that $u_{0}= \phi_{0}+ \psi_{0},$
\begin{eqnarray}\label{ipt}
\|\psi_{0}\|_{M^{p_{0},p_{0}'}}\leq C \|u_{0}\|
_{M^{p,p'}}\frac{1}{N}, \quad \|\phi_{0}\|_{L^2}\leq C \|u_{0}\|_{M^{p,p'}}N^{\beta}
\quad \text{and}\quad \beta=\frac{\frac{1}{p}-\frac{1}{2}}{\frac{1}{p_{0}}-\frac{1}{p}}. 
\end{eqnarray} 
\item[--] \emph{(Interpolation of $ \widehat{L}^{p}-$spaces with $p>2$, \cite[Section 4, Theorem 2(ii)]{hyakuna2012existence}).} Let $u_{0} \in \widehat{L}^{p},\; p\in (2,p_{1})$ and $N>0$.
 Then, there exist $\phi_{0} \in L^{2}$ and $\psi_{0} \in \widehat{L}^{p_{1}}$  such that $u_{0}= \phi_{0}+ \psi_{0},$
\begin{eqnarray}\label{ipt3}
 \|\psi_{0}\|_{\widehat{L}^{p_{1}}}\leq C \|u_{0}\|
_{\widehat{L}^{p}}\frac{1}{N}, \quad \|\phi_{0}\|_{L^2}\leq C \|u_{0}\|_{\widehat{L}^{p}}N^{\beta_{2}} \quad \text{and}\quad
\beta_{2}=\frac{\frac{1}{2} - \frac{1}{p}}{\frac{1}{p} - \frac{1}{p_{1}}}. 
\end{eqnarray}
\end{itemize}

\section{Nonlinear Estimates 
}\label{s3}
\par{Denote the \emph{generalised Strichartz spaces} by
\begin{equation}\label{X(T)}
    X(T):=L^{Q_{p_{0}}(r)}_{T}L^{r}, \quad (Q_{p_{0}}(r),r) \in \widehat{\mathcal{X}}(p_{0})
\end{equation} 
and the \emph{Strichartz spaces} by
\begin{equation}\label{Y(T)}
    Y(T):=L^{q(r)}_{T}L^{r}, \quad  (q(r),r)\in \mathcal{A} .
\end{equation}}
\begin{Remark}\label{whyp0so} Recall the parameter $r$ given in \eqref{rvalues}. In order for $(Q_{p_{0}}(r),r)\in \widehat{\mathcal{X}}(p_{0})$, we impose a suitable lower bound on $p_{0}$ as given in \eqref{pminloc} in each of the two cases corresponding to the respective values of $r$.

\begin{itemize}
        \item[-] Note that when $\alpha  \in (1,3)$ and  $r=\alpha+1$, we impose the condition $(\alpha+1)/\alpha<p_{0}$ so that $(Q_{p_{0}}(\alpha+1),\alpha+1)\in \widehat{\mathcal{X}}(p_{0})$.  In this case, the values of $1/p_{0}$ and $1/r$ fall within the region $R_1$ (cyan) of Figure \ref{fig2:myplot}.
        \item[-] While for $\alpha \in [3, 5-2b)$ and $r=2(\alpha-1)$, we impose the condition $2(\alpha-1)/\alpha \leq p_{0}$ so that $(Q_{p_{0}}(2(\alpha-1)),2(\alpha-1))\in \widehat{\mathcal{X}}(p_{0})$.  The range of $1/p_{0}$ and $1/r$ in this case can be presented by the red line or region $R_{2}$ (green) of Figure \ref{fig2:myplot}.
\end{itemize}
\end{Remark}
\begin{Remark}\label{chipAp}
   For $4/3<p_0<2$ and $2<r$, we have $Q_{p_{0}}(r)<q(r)$. Thus, for $T\leq 1,$ we have $$Y(T) \hookrightarrow X(T).$$
\end{Remark}
\begin{lem}\label{lemlwp1}
    Let $\alpha \in (1,5-2b),\, 0<b\leq 1/4$ and $p_{0}$  be as follows 
     \begin{gather*}
2>p_{0}>\begin{cases}
\vspace{0.2cm}
\max \left( \frac{\alpha+1}{\alpha}, \frac{4\alpha}{9-4b} \right) &\text{if} \quad \alpha \in  (1,3) \\ 
\vspace{0.2cm}
\max \left( \frac{2(\alpha-1)}{\alpha}, \frac{2\alpha}{5-2b} \right) \quad  & \text{if} \quad \alpha\in [3,5-2b).
\end{cases}
    \end{gather*}
    Then
    \begin{align*}
   \inf_{(\gamma,\rho)\in \mathcal{A}}\| ~|x|^{-b}|u|^{\alpha-1}v\|_{L^{\gamma'}_{T}L^{\rho '}}
   &\lesssim \|u\|_{X(T)}^{\alpha-1}\|v\|_{X(T)}
   \left(
     T^{1-\frac{\alpha-1}{2p_{0}}-\frac{1}{2}\left(\frac{1}{p_{0}}-\frac{1}{2}\right)} 
     + T^{1-\frac{\alpha-1}{2p_{0}}-\frac{1}{2}\left(\frac{1}{p_{0}}-\frac{1}{2}\right)-\frac{1}{\nu}}
\right).
\end{align*}
   
\end{lem}

\begin{proof}
Consider 
\begin{align*}
   & \inf_{(\gamma,\rho)\in \mathcal{A}} \|~|x|^{-b}|u|^{\alpha-1}v\|_{L^{\gamma'}_{T}L^{\rho '}}\\
    &\leq \inf_{(\gamma,\rho)\in \mathcal{A}}\|~|x|^{-b}|u|^{\alpha-1}v\|_{L^{\gamma'}_{T}L^{\rho '}(B^{c})} +\inf_{(\gamma,\rho)\in \mathcal{A}}\|~|x|^{-b}|u|^{\alpha-1}v\|_{L^{\gamma'}_{T}L^{\rho '}(B)} \\
    & = A_{1}+A_{2}.
\end{align*}
We need to find an admissible pair \footnote{Refer to \eqref{aps} for the definition of admissible pairs.} $(\gamma_{1},\rho_{1})$  to estimate $A_{1}$ by $X(T)$ norm. Using H\"older's inequality twice, we obtain
    \begin{align}
        A_{1} &\lesssim \|~|u|^{\alpha-1}v\|_{L^{\gamma'_{1}}_{T}L^{\rho '_{1}}(B^{c})} \nonumber\\
        &\label{YTB1}\leq T^{\frac{1}{\theta}}\|u\|_{X(T)}^{\alpha-1}\|v\|_{X(T)} 
    \end{align}
 with $\rho'_{1}$ satisfying H\"older conditions
 $$
      \frac{1}{\rho'_{1}}=\frac{\alpha-1}{r}+\frac{1}{r}.
$$
   This gives value of $\rho'_{1}=r/ \alpha $ \footnote{Note that $\rho' \in [1,2]$ in all three cases.} (and $\rho_{1}=r/ (r-\alpha))$. Since $(\gamma_{1},\rho_{1})$ is an admissible pair, we get $$\gamma'_{1}=\frac{4r}{5r-2\alpha}.$$ 
  While $\gamma'_{1}$ satisfies H\"older conditions 
  \begin{equation}\label{H1t}
  \frac{1}{\gamma'_{1}}=\frac{5}{4}-\frac{\alpha}{2r}=\frac{1}{\theta}+\frac{\alpha-1}{Q_{p_0}}+\frac{1}{Q_{p_{0}}}.
  \end{equation}
Here $1/\theta$ represents the exponent of $T.$ Solving for $1/\theta$, we get $$\frac{1}{\theta}=1-\frac{\alpha-1}{2p_{0}}-\frac{1}{2}\left(\frac{1}{p_{0}}-\frac{1}{2}\right).$$ 
 Note that $1/ \theta$ is positive by the hypothesis on $p_{0}>2\alpha /5$. Refer to Remark \ref{Remarkwhyp0so} \eqref{plemma}.\\
Similarly, we need to find an admissible pair $(\gamma_{2},\rho_{2})$ to estimate $A_{2}$ by $X(T)$ norm. Applying H\"older's inequality twice, we obtain
   \begin{align}
        \hspace{1.6cm}A_{2}&\leq \|~|x|^{-b}|u|^{\alpha-1}v \|_{L^{\gamma'_{2}}_{T}L^{\rho '_{2}}(B)} \nonumber\\
        &\label{A2E1}\leq \| ~\| |x|^{-b}\|_{L^{\nu}(B)} \|u\|^{\alpha-1}_{L^{r}}\|v\|_{L^{r}}\|_{L^{\gamma'_{2}}_{T}} \\
        &\label{A2E2}\leq  T^{\frac{1}{q_{1}}} \|~|x|^{-b}\|_{L^{\nu}(B)} \|u\|^{\alpha-1}_{X(T)}\|v\|_{X(T)}.
    \end{align}
Here $\rho'_{2},\gamma'_{2}, q_{1}$ and $\gamma_{3}$ satisfy the following conditions:
    \begin{eqnarray}
       \label{HXV} \frac{1}{\rho'_{2}} &=& \frac{1}{\nu}+\frac{\alpha-1}{r}+\frac{1}{r}\\
        \label{HTV}\frac{1}{\gamma'_{2}} &=& \frac{1}{q_{1}}+\frac{ \alpha-1}{Q_{p_0}}+\frac{1}{Q_{p_0}}\\
        \label{AP}\frac{5}{2}&=&\frac{2}{\gamma'_{2}} +\frac{1}{\rho'_{2}} \\
        \label{tpower}\frac{1}{q_{1}}& >&0 \\
       \label{SLB} \frac{1}{\nu}&>&b.
         \end{eqnarray}
        Note that \eqref{HXV} and \eqref{HTV} are due to H\"older's inequality for space and time variables applied to get \eqref{A2E1} and \eqref{A2E2} respectively. Since $(\gamma_{2},\rho_{2})$ is an admissible pair,  we have \eqref{AP}. The exponent of $T$ needs to be positive in \eqref{A2E2}, hence we have \eqref{tpower}. Condition \eqref{SLB} is required as $\|~|x|^{-b}\|_{L^{\nu}(B)} <\infty$ if and only if $\frac{1}{\nu}>b.$
         \\
         From \eqref{HXV} and \eqref{SLB},
\begin{equation}\label{DC}
    \frac{1}{\rho'_{2}}-\frac{\alpha}{r}>b, \quad i.e., \   \alpha < r \left(\frac{1}{\rho'_{2}}-b\right).
\end{equation}
It is important to note that we can not run the whole argument with the same value of $r=\alpha+1$ throughout $\alpha \in (1, 5-2b)$. In that case, we are not able to find an admissible pair $(\gamma_{2}, \rho_{2}) \in \mathcal{A}$ for $\alpha \in (3, 5-2b)$. 
In view of this, here we shall consider two cases.
\begin{enumerate}
    \item \textbf{Case 1:} $1<\alpha< 3$ and $r=\alpha+1$. \\
    In this case by \eqref{DC}, we have \begin{align*} 
\alpha &<(\alpha+1) \left(\frac{1}{\rho'_{2}}-b\right)\\
\alpha&< \frac{(1-b)(1-b\rho_{2}')}{\rho_{2}'(1+b)-1}
 \end{align*}
     Thus, solving for $\rho'_{2}$ 
\begin{equation*}
  \frac{1-b\rho_{2}'}{\rho_{2}'(1+b)-1} =3
\end{equation*} yields 
\begin{equation}\label{rho'1}
    \rho'_{2}=\frac{4}{4b+3}.
\end{equation}
Inserting the value of $\rho'_{2}$ in \eqref{AP}, we get
\begin{equation}
    \frac{1}{\gamma'_{2}}=\frac{5}{4}-\frac{4b+3}{8}.
\end{equation}
From \eqref{HTV}, we obtain 
\begin{align}
    \frac{1}{q_{1}}&=\frac{5}{4}-\frac{4b+3}{8}-\frac{\alpha}{Q_{p_{0}}} \nonumber\\
    &=\frac{5}{4}-\frac{3}{8}-\frac{\alpha}{2p_{0}}+\frac{\alpha}{2(\alpha+1)}-\frac{b}{2} \nonumber\\
    &=1-\frac{\alpha-1}{2p_{0}}-\frac{1}{2}\left(\frac{1}{p_{0}}-\frac{1}{2}\right)-\left(\frac{3}{8}-\frac{\alpha}{2(\alpha+1)}\right)-\frac{b}{2}\nonumber\\
     &=1-\frac{\alpha-1}{2p_{0}}-\frac{1}{2}\left(\frac{1}{p_{0}}-\frac{1}{2}\right)-\left(\displaystyle \sup_{\alpha \in (1,3)}  \frac{\alpha}{2r}-\frac{\alpha}{2r}\right)-\frac{b}{2}\label{q11}.
    \end{align}
\item \textbf{Case 2:} $3\leq \alpha < 5-2b$ and $r=2(\alpha-1)$.\\
In this case by \eqref{DC} we have \begin{align*}
\alpha &<2(\alpha-1) \left(\frac{1}{\rho'_{2}}-b\right)\\
\alpha&< \frac{2(b\rho_{2}'-1)}{\rho_{2}'(1+2b)-2}
 \end{align*}
Solving for $\rho'_{2}$ gives
\begin{equation*}
  \frac{2(b\rho'_{2} -1)}
  {\rho'_{2}+2b\rho'_{2}-2}=5-2b
\end{equation*} yields 
\begin{equation}\label{rho'2}
   \rho'_{2}=\frac{4(2-b)}{5+6b-4b^{2}}.
\end{equation}
Inserting the value of $\rho'_{2}$ in \eqref{AP}, we get
\begin{equation}
    \frac{1}{\gamma'_{2}}=\frac{5}{4}-\frac{5+6b-4b^2}{8(2-b)} .
\end{equation}
From \eqref{HTV}, we obtain 
\begin{align}
    \frac{1}{q_{1}}&=\frac{5}{4}-\frac{5+6b-4b^2}{8(2-b)}-\frac{\alpha}{Q_{p_{0}}} \nonumber\\
    &=\frac{5}{4}-\frac{5-2b}{8(2-b)}-\frac{\alpha}{2p_{0}}+\frac{\alpha}{4(\alpha-1)}-\frac{b}{2} \nonumber\\
     &=1-\frac{\alpha-1}{2p_{0}}-\frac{1}{2}\left(\frac{1}{p_{0}}-\frac{1}{2}\right)-\left(\displaystyle \inf_{\alpha \in [3,5-2b)}\frac{\alpha}{4(\alpha-1)}  -\frac{\alpha}{4(\alpha-1)}\right)-\frac{b}{2}\nonumber\\
      &=1-\frac{\alpha-1}{2p_{0}}-\frac{1}{2}\left(\frac{1}{p_{0}}-\frac{1}{2}\right)-\left(\displaystyle \inf_{\alpha \in [3,5-2b)}  \frac{\alpha}{2r}-\frac{\alpha}{2r}\right)-\frac{b}{2}\label{q12}.
    \end{align}
    \end{enumerate}
Considering \eqref{q11} and \eqref{q12} and further, substituting the value of $ 1 / q_{1} $  in \eqref{A2E2}, we have \begin{align}
 A_{2} &\label{YTB2}\lesssim T^{1-\frac{\alpha-1}{2p_{0}}-\frac{1}{2}\left(\frac{1}{p_{0}}-\frac{1}{2}\right)-\frac{1}{\nu}}\|u\|^{\alpha-1}_{X(T)}\|v\|_{X(T)}.
\end{align} 
Combining \eqref{YTB1} and \eqref{YTB2}, we have the claim.  
\end{proof}  
\begin{Remark}\label{Remarkwhyp0so} We have the following remark to make on Lemma \ref{lemlwp1}.
\begin{enumerate}
\item \label{whybso} Note that $
 \rho'_{2}$ for both cases (\eqref{rho'1} and \eqref{rho'2})  are  in $ \in [1,2]$ for $b \leq 1/4$ so that $(\gamma_{2}, \rho_{2}) \in \mathcal{A}$.
    \item \label{positiveexprem}
The positivity of the exponent of $T$ in both cases follows directly from the fact that 
    \begin{eqnarray*}
     2>p_{0}> 
        \begin{cases}
        \vspace{0.2cm}
          \frac{4\alpha}{9-4b} &\text{if} \; \alpha \in (1,3)\\
         \frac{2\alpha}{5-2b} &\text{if} \; \alpha \in [3,5-2b).
        \end{cases}
    \end{eqnarray*}
\item
    The lower bound on $p_{0}$ in Lemma \ref{lemlwp1} is due to Remark \ref{whyp0so} and  Remark \ref{Remarkwhyp0so}\eqref{positiveexprem}. This gives the lower bound on $p_{0}$ as
    \begin{gather*}
2>p_{0}>\begin{cases}
\vspace{0.2cm}
\max \left( \frac{\alpha+1}{\alpha}, \frac{4\alpha}{9-4b} , \frac{2\alpha}{5}\right) &\text{if} \quad \alpha \in  (1,3) \\ 
\vspace{0.2cm}
\max \left( \frac{2(\alpha-1)}{\alpha}, \frac{2\alpha}{5-2b} ,\frac{2\alpha}{5}\right) \quad  & \text{if} \quad \alpha\in [3,5-2b).
\end{cases}
    \end{gather*}
\item  \label{plemma}   Since $\frac{4\alpha}{9-4b} > \frac{2\alpha}{5}$ for $\alpha \in (1,3)$ and $\frac{2\alpha}{5-2b}>\frac{2\alpha}{5}$ for $\alpha \in [3, 5-2b)$, we can simply write the above hypothesis on the lower bound of $p_{0}$ as
    \begin{gather*}
2>p_{0}>\begin{cases}
\vspace{0.2cm}
\max \left( \frac{\alpha+1}{\alpha}, \frac{4\alpha}{9-4b} \right) &\text{if} \quad \alpha \in  (1,3) \\ 
\vspace{0.2cm}
\max \left( \frac{2(\alpha-1)}{\alpha}, \frac{2\alpha}{5-2b} \right) \quad  & \text{if} \quad \alpha\in [3,5-2b).
\end{cases}
    \end{gather*}
\end{enumerate}
\end{Remark}

\section{Well-posedness in \texorpdfstring{$\widehat{L}^{p}$}{Lp-hat} and \texorpdfstring{$M^{p,p'}$}{Mp} for \texorpdfstring{$p<2$}{p<2}}\label{s4}
In this section, we present proofs of Theorems \ref{lwp2}-\ref{mr2}.
\begin{proof}[\textbf{Proof of Theorem \ref{lwp2}}]
Based on the Duhamel principle, the solution to \eqref{INLS} is equivalent to the integral equation 
\begin{equation}\label{DP}
    u(t)=e^{it \partial_x^2}u_{0}+i \mu \int_{0}^{t}e^{i(t-s)\partial_{x}^{2}}|x|^{-b}(|u|^{\alpha-1}u)(s)ds:=\Lambda(u)(t).
\end{equation}
Let $a$ and $T$ be positive real numbers (to be chosen later), and $X(T)$ be as in \eqref{X(T)}. 
    Define $$B(a,T):=\{u\in X(T):\|u\|_{X(T)}\leq a\}.$$
We will show that $\Lambda$ 
    is a contraction map on $B(a,T).$ 
    Assume w.l.o.g. that $T\leq 1$.  Firstly, we consider the linear evolution of $u_{0} \in \widehat{L}^{p_{0}}$. Using \eqref{gse1}, we have
    \begin{align}
        \|e^{i t \partial_x^2}u_{0}\|_{X(T)}
         &\label{linearu0}\lesssim      \|u_{0}\|_{\widehat{L}^{p_{0}}}.
    \end{align}
    This suggests the choice of
    $
    a=C(\alpha,b)\|u_{0}\|_{\widehat{L}^{p_{0}}}.$
    Using  Remark \ref{chipAp}, \eqref{st} and Lemma \ref{lemlwp1}, for $u\in B(a,T)$, we have
    \begin{align}
    \left|\left| \int_{0}^{t} e^{i(t-\tau)\partial_x^2}|x|^{-b}(|u|^{\alpha-1}u)(s)ds \right|\right|_{X(T)}
    &\lesssim \left|\left| \int_{0}^{t} e^{i(t-\tau)\partial_x^2}|x|^{-b}(|u|^{\alpha-1}u)(s)ds \right|\right|_{Y(T)} \nonumber\\
    &\lesssim  \inf_{(\gamma,\rho)\in \mathcal{A}} \|~|x|^{-b}|u|^{\alpha-1}v\|_{L^{\gamma'}_{T}L^{\rho '}} \nonumber\\
    &\leq \inf_{(\gamma,\rho)\in \mathcal{A}}\|~|x|^{-b}|u|^{\alpha-1}v\|_{L^{\gamma'}_{T}L^{\rho '}(B^{c})} \nonumber\\ 
    +&\inf_{(\gamma,\rho)\in \mathcal{A}}\|~|x|^{-b}|u|^{\alpha-1}v\|_{L^{\gamma'}_{T}L^{\rho '}(B)}\nonumber \\
   &\lesssim_{\alpha,b} \||u|^{\alpha-1}u\|_{L^{\gamma'_{1}}_{T}L^{\rho '_{1}}(B^{c})}+\|~|x|^{-b} |u|^{\alpha-1}u\|_{L^{\gamma'_{2}}_{T}L^{\rho '_{2}}(B)}\nonumber\\
    &\label{ipGamma2}\lesssim   T^{1-\frac{\alpha-1}{2p_{0}}-\frac{1}{2}\left(\frac{1}{p_{0}}-\frac{1}{2}\right)-\frac{1}{\nu}}\|u\|_{X(T)}^{\alpha}.
    \end{align}
    Taking
    \begin{equation*}
        T :=\min \left\{ 1,~C(\alpha,b) \|u_{0}\|^{-\frac{\alpha-1}{1-\frac{\alpha-1}{2p_{0}}-\frac{1}{2}\left(\frac{1}{p_{0}}-\frac{1}{2}\right)-\frac{1}{\nu}}}_{ \widehat{L}^{p_{0}}} \right \},
    \end{equation*}
 we combine \eqref{linearu0} and \eqref{ipGamma2} to conclude $\Lambda(u)\in B(a, T).$\\
 Using  \eqref{eg} for $G(0,u,v)$ and the previous argument employed, for $u,v \in B(a,T)$, we have
\begin{align}
\hspace{-4cm}\|\Lambda(u)-\Lambda(v)\|_{X(T)} & \lesssim \left|\left| \int_{0}^{t} e^{i(t-\tau)\partial_x^2}|x|^{-b}G(0,u,v)(\tau)d\tau \right|\right|_{Y(T)} \nonumber
\end{align}
\begin{align}
&\lesssim_{\alpha,b} \|(|u|^{\alpha-1}+|v|^{\alpha-1})|u-v|\|_{L^{\gamma'_{1}}_{T}L^{\rho '_{1}}(B^{c})}+\|~|x|^{-b}(|u|^{\alpha-1}+|v|^{\alpha-1})|u-v|\|_{L^{\gamma'_{2}}_{T}L^{\rho '_{2}}(B)} \nonumber \\
    &\lesssim_{n} T^{1-\frac{\alpha-1}{2p_{0}}-\frac{1}{2}\left(\frac{1}{p_{0}}-\frac{1}{2}\right)-\frac{1}{\nu}}\left\{\|u\|_{X(T)}^{\alpha-1}+\|v\|_{X(T)}^{\alpha-1}\right\}\|u-v\|_{X(T)}.
\end{align}
By the choice of $a$ and $T$ \footnote{In the expression of $a$ and $T, C(\alpha,b)$ is chosen sufficiently small to ensure that $\Lambda(u)$ is a contraction map on $B(a, T).$}, we have $$\|\Lambda(u)-\Lambda(v)\|_{X(T)}\leq \frac{1}{2}\|u-v\|_{X(T)}.$$
Thus, by the contraction mapping theorem, we obtain a unique fixed point of $\Lambda(u)$, which is a solution to \eqref{INLS}.
\end{proof}
\begin{proof}[\textbf{Proof of Corollary \ref{lwp}}]
Let $u_{0}\in M^{p_{0},p_{0}'}$. To prove Theorem \ref{lwp}, we only need to estimate the linear part, as the integral part can be treated similarly as done for Theorem \ref{lwp2}. \\
Considering \eqref{gse1} and  \eqref{embdModFLS}, we have
  \begin{align*}
  \|e^{i t \partial_x^2}u_{0}\|_{X(T)}&\lesssim      \|u_{0}\|_{\widehat{L}^{p_{0}}}\\
         &\lesssim      \|u_{0}\|_{M^{p_{0},p_{0}'}}.
    \end{align*} 
    The rest of the proof follows the same general lines as the proof of Theorem \ref{lwp2}
from the previous section.

\end{proof}
\begin{proof}[\textbf{Proof of Theorem \ref{mr2}}]
We start with decomposing initial data $u_0 \in \widehat{L}^{p} \subset \widehat{L}^{p_{0}}+L^{2} $ into two parts using \eqref{ipt2}. 
Specifically, for some $N>0$ and $u_0 \in \widehat{L}^{p} $, there exist $\phi_{0} \in L^{2}$ and $\psi_{0} \in \widehat{L}^{p_{0}}$  such that 
\begin{eqnarray}\label{decomposition}
\begin{cases}
    u_{0}= \phi_{0}+ \psi_{0}, \\ \|\psi_{0}\|_{\widehat{L}^{p_{0}}}\leq C \|u_{0}\|
_{\widehat{L}^{p}} \frac{1}{N}, \\ 
\|\phi_{0}\|_{L^2}\leq C \|u_{0}\|_{\widehat{L}^{p}}N^{\beta_{1}},
\end{cases}
\end{eqnarray}
where
\begin{equation}\label{betavalue}
    \beta_{1}=\frac{\frac{1}{p}-\frac{1}{2}}{\frac{1}{p_{0}}-\frac{1}{p}}.
\end{equation}
Consider  \eqref{INLS} with initial data $\phi_0:$
\begin{eqnarray}\label{ivpL2}
\begin{cases}
 i \partial_t v_{0} + \partial_x^2 v_{0}+\mu |x|^{-b}|v_{0}|^{\alpha-1}v_{0}=0 \\
v_{0}(\cdot,0)=\phi_{0}\in L^2.  \\
\end{cases}        
    \end{eqnarray}
We recall that, in  \cite[Theorem 1.8]{guzman2020},  Guzm\'an proved that \eqref{ivpL2} has a unique global solution $v_{0}$ satisfying
\begin{eqnarray}\label{gwpl2}
\begin{cases}
  v_{0}\in C(\R, L^2) \cap L_{loc}^{q}(\R,L^{r} ),\\
\sup_{(q,r)\in \mathcal{A}}\|v_{0}\|_{L_{loc}^{q} L^{r}} \lesssim_{r} \|\phi_{0}\|_{L^{2}}.
\end{cases}
\end{eqnarray}
Now, consider the modified \eqref{INLS} associated with the evolution of $\psi_{0}$: 
\begin{eqnarray}\label{ivpMod}
\begin{cases}
 i \partial_t w + \partial_x^2 w +\mu |x|^{-b}(|w+v_{0}|^{\alpha-1}(w+v_{0})-|v_{0}|^{\alpha-1}v_{0})=0 \\
w(\cdot,0)=\psi_{0}\in \widehat{L}^{p_{0}} .
\end{cases}
\end{eqnarray}
The solution of the above I.V.P. \eqref{ivpMod} is given as
\begin{equation}\label{modsoln1}
e^{i t \partial_x^2}\psi_{0}+w_{0}.
\end{equation}
The  nonlinear interaction $w_0$ corresponding to $\psi_0$ can be expressed as \begin{align}\label{modsoln2}
    w_{0} = i \mu \int_{0}^{t}  e^{i(t-s)\partial_x^2}|x|^{-b} \left( \big| e^{is \partial_x^2}\psi_{0}+w_{0}+v_{0}\big|^{\alpha-1} (e^{is \partial_x^2}\psi_{0}+w_{0}+v_{0}) - |v_{0}|^{\alpha-1} v_{0}\right)\, ds.
\end{align}
Formally,  we may rewrite  the  solution to  \eqref{INLS} corresponding to data $u_{0}$ as follows \begin{eqnarray}\label{solnlocal}
u= v_{0}+e^{i t \partial_x^2}\psi_{0}+w_{0}.
\end{eqnarray}
 In view of \eqref{gwpl2}, \eqref{spfls} and \eqref{gse1},   we notice that  $v_0$ and $e^{i t \partial_x^2} \psi_0$ are globally defined in appropriate spaces.  To establish the desired global existence, we first determine the time interval of existence for $w_0$.   To this end, we may rewrite
\begin{align}\label{w01}
    w_{0} 
    =i \mu \int_{0}^{t} e^{i(t-s)\partial_x^2}|x|^{-b} G(v_{0} + e^{is\partial_x^2}\psi_{0}, w_{0},0) \, ds +  i \mu \int_{0}^{t} e^{i(t-s)\partial_x^2}|x|^{-b} G(v_{0}, e^{is\partial_x^2}\psi_{0},0) \, ds,
\end{align}
 with $G$ as defined in \eqref{eg}. We have the following local well-posedness result for the integral equation \eqref{w01}.
 \begin{prop}\label{w0exist} Let $\phi_{0} \in L^2 $ and $ \psi_{0} \in \widehat{L}^{p_{0}}$. Denote by  $v_{0}$  the $L^2-$ global solution  (as in \eqref{gwpl2}) for initial value $\phi_{0}$.
Then  there exists  a constant $C=C(\alpha,b)>0$ such that  the integral equation $$w_{0}= i \mu \int_{0}^{t} e^{i(t-s)\partial_x^2}|x|^{-b} G(v_{0} + e^{is\partial_x^2}\psi_{0}, w_{0},0) \, ds +  i \mu \int_{0}^{t} e^{i(t-s)\partial_x^2}|x|^{-b} G(v_{0}, e^{is\partial_x^2}\psi_{0},0) \, ds $$ has a unique solution  $w_{0} \in Y(T)$   provided $T$ satisfying 
\begin{align}
    \label{c1} T &\leq 1 \\
     \label{c2} T &\leq C \left( \|\phi_{0}\|_{L^{2}} + \|\psi_{0}\|_{\widehat{L}^{p_{0}}} \right)^{-\frac{\alpha-1}{1-\frac{\alpha-1}{2p_{0}}-\frac{1}{\nu}} }\\ 
     \label{c3} T &\leq C \left( \|\psi_{0}\|_{\widehat{L}^{p_{0}}} \right)^{-\frac{\alpha-1}{1-\frac{\alpha-1}{2p_{0}}-\frac{1}{\nu}}}.
\end{align}
\end{prop}
\begin{proof}
     Define $$B(A,T)=\{w_{0}\in Y(T):\|w_{0}\|_{Y(T)}\leq A\}$$
    with $A>0$ (to be chosen later), $T$ be the minimum of the right-hand sides of the conditions \eqref{c1},  \eqref{c2} and \eqref{c3} (w.l.o.g. we may assume). 
    Further, define
    $$\Gamma(w_{0}):= i \mu \int_{0}^{t} e^{i(t-s)\partial_x^2}|x|^{-b} G(v_{0} + e^{is\partial_x^2}\psi_{0}, w_{0},0)\, ds +  i \mu  \int_{0}^{t} e^{i(t-s)\partial_x^2}|x|^{-b} G(v_{0}, e^{is\partial_x^2}\psi_{0},0) \, ds. $$
Firstly, we need to show that $\Gamma(w_{0})\in B(A,T)$. Under the assumption $\eqref{c1}$ and applying \eqref{st}, Lemma \ref{lemlwp1}, \eqref{gwpl2} and \eqref{gse1}, we obtain 
    \begin{align*}\left|\left| \displaystyle\int_{0}^{t} e^{i(t-s)\partial_x^2} |x|^{-b}G(v_{0},e^{is \partial_x^2}\psi_{0},0) \, d\tau \right|\right|_{Y(T)} 
   \lesssim_{\alpha,b} & \||v_{0}|^{\alpha-1}e^{is\partial_x^2}\psi_{0}\|_{L^{\gamma'_{1}}_{T}L^{\rho'_{1}}(B^{c})}    \nonumber\\  
   +  \|~| x|^{-b}|v_{0}|^{\alpha-1}e^{i\tau \partial_x^2}\psi_{0}\|_{L^{\gamma'_{2}}_{T}L^{\rho'_{2}}(B)}  
   +& \||e^{is \partial_x^2}\psi_{0} |^{\alpha} \|_{L^{\gamma'_{1}}_{T}L^{\rho'_{1}}(B^{c})} 
   + \|~|x|^{-b}|e^{is \partial_x^2}\psi_{0} |^{\alpha} \|_{L^{\gamma'_{2}}_{T}L^{\rho'_{2}}(B)} 
    \end{align*}
    \begin{align*}
   \lesssim_{\alpha,b}
   & \left(T^{1-\frac{\alpha-1}{4}-\frac{1}{2}\left(\frac{1}{p_{0}}-\frac{1}{2}\right)-\frac{1}{\nu}} \|v_{0}\|_{Y(T)}^{\alpha-1}\| e^{i\tau \partial_x^2}\psi_{0}\|_{X(T)}+T^{1-\frac{\alpha-1}{2p_{0}}-\frac{1}{2}\left(\frac{1}{p_{0}}-\frac{1}{2}\right)-\frac{1}{\nu}}\| e^{i\tau \partial_x^2}\psi_{0}\|^{\alpha}_{X(T)}\right)\\
   \lesssim_{\alpha,b}
   & T^{-\frac{1}{2}\left(\frac{1}{p_{0}}-\frac{1}{2}\right)} \| \psi_{0}\|_{\widehat{L}^{p_{0}}} \left\{ T^{1-\frac{\alpha-1}{4}-\frac{1}{\nu}}\|\phi_{0}\|_{L^2}^{\alpha-1}+T^{1-\frac{\alpha-1}{2p_{0}}-\frac{1}{\nu}}\| \psi_{0}\|_{\widehat{L}^{p_{0}}}^{\alpha-1} 
  \right \}\\
    \lesssim_{\alpha,b}
    & T^{-\frac{1}{2}\left(\frac{1}{p_{0}}-\frac{1}{2}\right)}  \| \psi_{0}\|_{\widehat{L}^{p_{0}}}.
\end{align*}
The last inequality follows due to our assumptions \eqref{c2} and \eqref{c3}.
This suggests the choice of
    \begin{equation}\label{AA}
    A= \frac{3}{C(\alpha,b)} T^{-\frac{1}{2}\left(\frac{1}{p_{0}}-\frac{1}{2}\right)}  \| \psi_{0}\|_{\widehat{L}^{p_{0}}}.
    \end{equation}
    where $C=C(\alpha,b)$ is the same constant as in \eqref{c2} and \eqref{c3}, chosen such that
    \begin{equation}\label{gammaw1}
      \left|\left| \displaystyle\int_{0}^{t} e^{i(t-s)\partial_x^2} |x|^{-b}G(v_{0},e^{is \partial_x^2}\psi_{0},0) \, ds\right|\right|_{Y(T)}\leq \frac{A}{3}.
    \end{equation}
Using \eqref{st}, Lemma \ref{lemlwp1},  \eqref{c1} and Remark \ref{chipAp}, we have
\begin{align*}
   \left\| \int_{0}^{t}  e^{i(t-s)\partial_x^2} |x|^{-b} G(v_{0} +  e^{is \partial_x^2}\psi_{0}, w_{0},0) \, ds \right\|_{Y(T)}  \lesssim_{\alpha,b} & \||v_{0}+ e^{is \partial_x^2}\psi_{0}|^{\alpha-1}w_{0}\|_{L^{\gamma'_{1}}_{T}L^{\rho'_{1}}(B^{c})}    \nonumber\\  
   +  \|~|x|^{-b}|v_{0}+ e^{i\tau \partial_x^2}\psi_{0}|^{\alpha-1}w_{0}\|_{L^{\gamma'_{2}}_{T}L^{\rho'_{2}}(B)}  
   +& \|| w_{0} |^{\alpha} \|_{L^{\gamma'_{1}}_{T}L^{\rho'_{1}}(B^{c})} 
   + \|~|x|^{-b}|w_{0} |^{\alpha} \|_{L^{\gamma'_{2}}_{T}L^{\rho'_{2}}(B)} 
    \end{align*}
\begin{align*}
     \lesssim_{\alpha,b}& \left( T^{1-\frac{\alpha-1}{2p_{0}}-\frac{1}{\nu}} \|v_{0}+ e^{i\tau \partial_x^2}\psi_{0}\|_{X(T)}^{\alpha-1}+  T^{1-\frac{\alpha-1}{4}-\frac{1}{\nu}}\| w_{0}\|^{\alpha-1}_{Y(T)}\right)
 \| w_{0}\|_{Y(T)}\\
   \lesssim_{\alpha,b}&\left\{ 
  T^{1-\frac{\alpha-1}{2p_{0}}-\frac{1}{\nu}}\left( \|\phi_{0}\|_{L^2}+\| \psi_{0}\|_{\widehat{L}^{p_{0}}}\right)^{\alpha-1}+  T^{1-\frac{\alpha-1}{2p_{0}}-\frac{1}{\nu}}\| \psi_{0}\|_{\widehat{L}^{p_{0}}}^{\alpha-1}\right \} \|w_{0}\|_{Y(T)}\\
    & \lesssim_{\alpha,b}\|w_{0}\|_{Y(T)} \left\{\frac{1}{3} +\frac{1}{3} \right\}.
\end{align*}
To get the last second inequality, we have used \eqref{gwpl2} and \eqref{gse1} in the first summand, and substituted the norm of $w_{0}$ in $Y(T)$ by $A$ ($A$ given in \eqref{AA}) in the second summand.  The last inequality follows due to \eqref{c2} (choosing $C$ in \eqref{c2} small enough). 
Finally,
        \begin{equation}\label{gamma2}
            \left\| \int_{0}^{t} e^{i(t-\tau)\partial_x^2} |x|^{-b} G(v_{0} + e^{i\tau \partial_x^2}\psi_{0}, w_{0},0)\, d\tau \right\|_{Y(T)} \leq \frac{2A}{3}.
        \end{equation}
        Combining \eqref{gammaw1} and \eqref{gamma2}, we can say that $\Gamma(w_{0})$ belongs to $B(A, T).$ The contractivity of $\Gamma$ follows similarly. Thus, by the Banach fixed-point theorem, we get a unique fixed point $w_{0}$ to the integral equation \eqref{w01} on the time interval $[0,T]$.
\end{proof}  
Consequently, we have a local solution $ w_{0} \in L^{q(r)}_{T}L^{r}$ on the time interval $[0,T]$. Thus, based on \eqref{gwpl2}, Proposition \ref{w0exist}, \eqref{spfls} and \eqref{gse1},  we conclude $$v_0 + w_0 \in  L^{q(r)}_{T}L^{r} \cap L^{\infty}_{T}L^{2}$$ and 
$$e^{i t \partial_x^2} \psi_0 \in L^{\infty}_{T}\widehat{L}^{p_{0}} \cap L^{Q_{p_{0}}(r)}_{T}L^{r}  .$$
Next, we claim that the local solution  $u(t),\;t\in [0,T]$ obtained in the previous step can be extended to the interval $[T,2T]$. To do so, we need to control the Duhamel term, as shown in the next corollary.

 \begin{cor}\label{winfty2}
 Under the hypothesis of Proposition \ref{w0exist}, there exists a constant $C(\alpha,b)$ satisfying
    \begin{equation*}
\|w_{0}\|_{L^{\infty}_{T} L^{2}} 
\lesssim_{\alpha,b} T^{-\frac{1}{2}\left(\frac{1}{p_{0}}-\frac{1}{2}\right)} \| \psi_{0}\|_{\widehat{L}^{p_{0}}}.
    \end{equation*}

 \end{cor}
\begin{proof} The proof follows from the
Strichartz estimates (Proposition \ref{st}) by replacing the norm defined with respect to $Y(T)$ by  $L^{\infty}_{T}L^{2}$. Using an admissible pair $ (\infty,2)$ on the left-hand side  
and the same pairs on the right-hand side of \eqref{gammaw1} and \eqref{gamma2} in the proof of Proposition \ref{w0exist}, we obtain
\begin{align*}
\|w_{0}\|_{L^{\infty}_{T}L^{2}} \leq & \left|\left|\int_{0}^{t} e^{i(t-\tau)\partial_x^2}|x|^{-b} G(v_{0} + e^{i\tau\partial_x^2}\psi_{0}, w_{0},0) \, d\tau \right|\right|_{L^{\infty}_{T}L^{2}}\\
&+\left|\left|\int_{0}^{t} e^{i(t-\tau)\partial_x^2} |x|^{-b} G(v_{0}, e^{i\tau \partial_x^2}\psi_{0},0) \, d\tau \right|\right|_{L^{\infty}_{T}L^{2}}\\
&\lesssim_{\alpha,b} T^{-\frac{1}{2}(\frac{1}{p_{0}}-\frac{1}{2})}\|\psi_{0}\|_{\widehat{L}^{p_{0}}}.
\end{align*} 
This completes the proof of Corollary \ref{winfty2}.
\end{proof}
Denote the constant from Proposition \ref{w0exist} by $C=C(\alpha,b)$ and put
\begin{equation}\label{TN}
        T=T(N)=(3CN^{\beta_{1}})^{-\frac{\alpha-1}{1-\frac{\alpha-1}{4}-\frac{1}{\nu}} }.
    \end{equation}
Using  Proposition \ref{w0exist} for $\phi=\phi_{0}$ and $\psi=\psi_{0}$,  the solution of \eqref{INLS} exists in the interval $[0, T(N)]$ and it is of the form
\[u=v_0 + e^{i t \partial_x^2} \psi_0 + w_0,\]
as given in \eqref{solnlocal}. We wish to extend our solution to the interval $[T(N), 2T(N)]$ by a similar procedure, but with the new initial data as the sum of the following two functions:    
\[ \phi_{1}=v_{0}(T)+w_{0}(T) \quad \text{and} \quad \psi_{1} =e^{i t \partial_x^2}\psi_{0}.\]
More generally, we 
 wish to  extend our solution further by considering the following iterative procedure :
\begin{itemize}
    \item[--]We define $\phi_{k}$ and $\psi_k$ for $k \geq 1$ (for $k=0,\; \phi_{0}$ and $\psi_0$ are defined in \eqref{decomposition}) as follows:
    \begin{equation}\label{newidk}
 \phi_{k}= v_{k-1}(kT) + w_{k-1}(kT) \quad \text{and} \quad \psi_{k}= e^{ikT\partial_x^2}\psi_{0}, 
\end{equation}
where
\begin{eqnarray}
    \begin{aligned}
     w_{k-1}(kT)=i \mu \displaystyle \int_{0}^{kT} e^{i(kT-s)\partial_x^2} |x|^{-b}G(v_{k-1}+e^{iks \partial_x^2}\psi_{0},w_{k-1}) ds \nonumber \\
   + i \mu \displaystyle \int_{0}^{kT} e^{i(kT-s)\partial_x^2} |x|^{-b}G(v_{k-1},e^{iks \partial_x^2}\psi_{0}) ds.
  \label{wKT}
    \end{aligned}
       \end{eqnarray}
       \item[--] Let $T^* \leq \infty$ denote the maximal time of existence. Assume for $kT\leq T^*\footnote{Note that $(K-1)T\leq T^*.$ Since $T\leq 1,$ we will have $KT\leq T^{*}+1.$}, \phi=\phi_{k}$ and $\psi=e^{ikT\partial_x^2}\psi_{0}$, $T$ satisfy all three conditions \eqref{c1}, \eqref{c2} and \eqref{c3} of Proposition \ref{w0exist}, where $k \in \{0, 1,\cdots, K-1\}.$
       \item[--]Let $v_{k}$ be INLS evolution of  $\phi_k$, and by construction 
\begin{equation}\label{ss}
   u(\cdot, t)= v_{k} (\cdot, t-kT ) + w_{k}(\cdot, t-kT) + e^{it \partial_x^2} \psi_0,  \quad \text{if} \ t\in [kT, (k+1)T] 
\end{equation}
 defines a solution of \eqref{INLS} for $k \in \{0, 1,\cdots, K-1\}$.  
 \end{itemize}
 Suppose, to the contrary, that for $u_0 \in \widehat{L}^{p}$,  the solution is not global in time. Therefore, we have the maximal time $T^* < \infty$. In this case,  we shall produce a solution $u$  of \eqref{INLS} (as defined in \eqref{ss}), which will exist on a larger interval $[0, T_1]$ for $T_1> T^*.$ This will lead to a contradiction to the maximal time interval $[0, T^*).$\\ We shall show that the iterative process ends with  $KT >T^*.$ Since  $v_{K}$ and $e^{i t \partial_x^2} \psi_0$ are globally defined in appropriate spaces, we are left to handle the nonlinear interaction term $w_{K}$ to extend the solution at the $K$th iteration. To do this, we shall use Proposition \ref{w0exist} with $\phi=\phi_{K}$ and $\psi=e^{iKT\partial_x^2}\psi_{0}.$  

Note that $T=T(N)\to 0$ as $N\to \infty$ (see \eqref{TN}), and so 
the  smallness condition \eqref{c1} is satisfied independently of $k$ for large $N.$\\
 Using \eqref{spfls} and  \eqref{decomposition}, we have
\begin{align}\label{ref1} 
\|e^{i t \partial_x^2}\psi_{0}\|_{L^{\infty}([0,T^{*}+1],\widehat{L}^{p_{0}}) } =\|\psi_{0}\|_{\widehat{L}^{p_{0}}} \lesssim \frac{1}{N} \xrightarrow{N \to \infty} 0.
    \end{align}
Inserting $e^{ikT\partial_x^2}\psi_{0}$ in the right hand side of \eqref{c3}, we have
\begin{eqnarray*}
\begin{aligned}
    \left( \|e^{ikT\partial_x^2}\psi_{0}\|_{\widehat{L}^{p_{0}}} \right)^{-\frac{\alpha-1}{1-\frac{\alpha-1}{2p_{0}}-\frac{1}{\nu}}} 
    \gtrsim_{n} N^{\frac{\alpha-1}{1-\frac{\alpha-1}{2p_{0}}-\frac{1}{\nu}}} \xrightarrow{N \to \infty} \infty.
   \end{aligned}
\end{eqnarray*}  
Since the lower bound is independent of $k$ and $T \xrightarrow{N \to \infty} 0,$  condition \eqref{c3} holds for sufficiently large $N.$ \\

\noindent 
Thus, we either have $KT> T^*$ or condition \eqref{c2} fails in the last iterative step $k=K,$ i.e.
\begin{equation}\label{aim11}
    3CN^{\beta_{1}}< \|\phi_{K}\|_{L^2}+ \|e^{iKT\partial_x^2}\psi_{0}\|_{\widehat{L}^{p_{0}}}.
\end{equation} 
Considering \eqref{ref1}, \eqref{aim11} can be written as
\begin{equation}\label{aim}
    3CN^{\beta_{1}}< \|\phi_{K}\|_{L^2}+CN^{\beta_{1}}.
\end{equation}
We claim that even under condition \eqref{aim},   we will have $KT>T^*$. This clearly leads to a contradiction to the definition of $ T^*.$ \\
\noindent
In view of the construction of $\phi_k$ and Corollary \ref{winfty2}, we note that $\phi_k \in L^2$ for $k\in \{0, 1,\cdots, K-1\}.$ Now exploiting the conservation of mass for \eqref{INLS} and Corollary \ref{winfty2} (for  $w=w_{k}$ and $\psi=e^{ikT\partial_x^2}\psi_{0}, 0\leq k \leq K-1$), we have
    \begin{align}
        \|\phi_{K}\|_{L^{2}}  &\leq \|v_{K-1}\|_{L^{\infty}_{[(K-1)T,KT]}L^{2}}+\|w_{K-1}\|_{L^{\infty}_{[(K-1)T,KT]}L^{2}} \nonumber\\
        &= \|\phi_{K-1}\|_{L^2} + \|w_{K-1}\|_{L^{\infty}_{[(K-1)T,KT]}L^2}\nonumber\\
        & \leq \|v_{K-2}\|_{L^{\infty}_{[(K-2)T,(K-1)T]}L^{2}}+\|w_{K-2}\|_{L^{\infty}_{[(K-2)T,(K-1)T]}L^{2}} +\|w_{K-1}\|_{L^{\infty}_{[(K-1)T,KT]}L^{2}} \nonumber\\
        &= \|\phi_{K-2}\|_{L^2}+\|w_{K-2}\|_{L^{\infty}_{[(K-2)T,(K-1)T]}L^{2}} +\|w_{K-1}\|_{L^{\infty}_{[(K-1)T,KT]}L^{2}} \nonumber\\
        & \leq \cdots \leq  \|\phi_{0}\|_{L^{2}}+\sum_{k=0}^{K-1}\|w_{k}\|_{L^{\infty}_{[kT,(k+1)T]}L^{2}} \nonumber \\
        & \lesssim_{\alpha,b} CN^{\beta_{1}}+T^{-\frac{1}{2}(\frac{1}{p_{0}}-\frac{1}{2})}\sum_{k=0}^{K-1}\|e^{ikT\partial_x^2}\psi_{0}\|_{\widehat{L}^{p_{0}}}\nonumber\\
        &\lesssim_{T^{*}} CN^{\beta_{1}}+T^{-\frac{1}{2}(\frac{1}{p_{0}}-\frac{1}{2})}K\frac{C}{N}\label{secondest}.
    \end{align}
    In the last two inequalities, we have used \eqref{decomposition} and \eqref{ref1}.
   Thus, using \eqref{TN}, \eqref{aim} can be expressed as
 \begin{align}
     KT &\gtrsim_{\alpha,b,T^*} N^{1+\beta_{1}}T^{1+\frac{1}{2}\left(\frac{1}{p_{0}}-\frac{1}{2}\right)} \approx N^{1+\beta_{1} \left(1-\frac{(\alpha-1)\left(1+\frac{1}{2}\left(\frac{1}{p_{0}}-\frac{1}{2}\right)\right)}{1-\frac{\alpha-1}{4}-\frac{1}{\nu}}\right)}\nonumber\\
    &\label{Npower}=N^{1-\beta_{1} \left(-1+\frac{(\alpha-1)\left(1+\frac{1}{2}\left(\frac{1}{p_{0}}-\frac{1}{2}\right)\right)}{1-\frac{\alpha-1}{4}-\frac{1}{\nu}}\right)}.
\end{align}
Note that $N$ can be taken arbitrarily large. For any $\beta$ satisfying 
\begin{align}\label{betarange}
0 < \beta_{1} <
\begin{cases}
 \eta \quad &\text{if}\quad  \alpha-2 +\frac{\alpha-1}{2p_{0}}+\frac{1}{\nu}>0\\
\infty \quad &\text{otherwise} 
\end{cases}
\end{align}
where
$$\eta =\frac{1-\frac{\alpha-1}{4}-\frac{1}{\nu}}
{ \alpha-2 +\frac{\alpha-1}{2p_{0}}+\frac{1}{\nu}},$$
the exponent of $N$ is positive in \eqref{Npower}, we get $KT>T^*$. This concludes the proof of Theorem \ref{mr2}.  
\begin{Remark}\label{etap}Recall $\beta_{1}$ defined in terms of $p$ and $p_{0}$ in  \eqref{betavalue}. The range of $\beta_{1}$ in \eqref{betarange} in turn decides the range of $p.$ 
    Note that when $$ \beta_{1}=\eta=\frac{1-\frac{\alpha-1}{4}-\frac{1}{\nu}}
{ \alpha-2 +\frac{\alpha-1}{2p_{0}}+\frac{1}{\nu}},$$ we get the lower bound of the admissible range of $p$, denoted by $p_{\min}$,
$$p_{\min}=\frac{(\alpha-1)(3p_{0}+2)}
{4 +2(\alpha-2)p_{0}-\frac{2}{\nu}(2-p_{0})}.$$
Thus, we have $p \in (p_{\min},2)$ and $p_{\min}$ explicitly given as:
     \begin{align}\label{pmin2}p_{\min}=
         \begin{cases}
             \frac{(\alpha-1)(3p_{0}+2)}
{4 +2(\alpha-2)p_{0}-\frac{2}{\nu}(2-p_{0})} &\text{if}\quad \beta_{1} = \eta \\
             p_{0} &\text{if}\quad \beta_{1}=\infty. \\
         \end{cases}
     \end{align}
      This justifies the admissible range of \(p\)  in Theorem \ref{mr}.
\end{Remark}
\end{proof}
 \begin{proof}[\textbf{Proof of Corollary \ref{mr}}]
 In the case of modulation spaces, the proof goes
along the same lines as above. Let $u_{0} \in M^{p,p'}.$ Using \eqref{ipt}, $u_{0}$ can be decomposed into $\phi_{0}\in L^2$ and $\psi_{0} \in M^{p_{0},p_{0}'}$. The solution $u_{0}$ corresponding to $\phi_{0}\in L^2$ is given as \eqref{gwpl2}. While the nonlinear part corresponding to $\psi_{0}$, $w_{0}\in Y(T)$, as claimed in Proposition \ref{w0exist}. 
Due to \eqref{Mpp1}, \eqref{gse1} and \eqref{embdModFLS},
the linear part corresponding to $\psi_{0}$ lies in
$$ e^{i t \partial_x^2} \psi_0 \in L^{\infty}_{T}M^{p_{0},p_{0}'} \cap L^{Q_{p_{0}}(r)}_{T}L^{r} .$$
Thus, the solution corresponding to $u_{0}$ to \eqref{INLS} lies in 
$$\{L^{\infty}_{T}L^2 \cap L^{q(r)}_{T}L^{r}\}+ \{L^{\infty}_{T}M^{p_{0},p_{0}'} \cap L^{Q_{p_{0}}(r)}_{T}L^{r}\}.$$
The rest follows similarly to the proof of Theorem \ref{mr2}.
\end{proof}
\section{Well-posedness in $\widehat{L}^{p}$ for $p>2$}\label{s5}
Consider the Banach space $X'(T)$ expressed as 
\begin{equation}
\begin{aligned}\label{X1(T)}
        X'(T)&:= X_{1}(T)+X_{2}(T)
\end{aligned}
\end{equation}
where $$X_1(T):=L^{\infty}_{T}L^2 ~\cap~ L^{q(r)}_{T} L^{r}$$
equipped with the norm 
$$\|v\|_{X_1(T)}=\max \left\{\|v\|_{L^{\infty}_{T}L^2},\|v\|_{L^{q(r)}_{T}L^{r}}\right\} $$ and $$ X_2(T):=L^{\infty}_{T}\widehat{L}^{r}.$$\\
The norm on $X'(T)$ is given as 
\begin{equation*}
\|u\|_{X'(T)} =\inf_{\substack{u=v+w \\ v \in X_1(T) \\ w \in X_2(T)}} \left(\|v\|_{X_1(T)} + \|w\|_{X_2(T)} \right).
\end{equation*}
Denote
\begin{equation}\label{Y1(T)}
Y_{1}(T):=L^{q(r)}_{T}L^{r}
    \end{equation}

\begin{lem}\label{lemlwp2}
    Let $1<\alpha<5-2b$ and $ 0<b\leq 1/4 .$
Then
$$\inf_{(\gamma,\rho)\in \mathcal{A}}\| ~|x|^{-b}|u|^{\alpha-1}v\|_{L^{\gamma'}_{T}L^{\rho '}} \lesssim  ( T^{\frac{5-\alpha}{4}}+ T^{\frac{5-\alpha}{4}-\frac{1}{\nu}})\|u\|_{Y_{1}(T)}^{\alpha-1}\|v\|_{Y_{1}(T)}.$$
\end{lem}

\begin{proof}
   The result follows from taking $p_{0}=2$ in Lemma \ref{lemlwp1}. Also refer to \cite[Lemma 3.1]{Bhimani2024low}, where the authors considered $r=\alpha+1$ for $\alpha \in (1,5-2b)$ and $0<b<(3-\sqrt{7})/2(<1/4)$.
\end{proof}
\begin{proof}[\textbf{Proof of Theorem \ref{lwp3}}]Consider 
\begin{equation*}
    u(t)=e^{it\partial_x^2}u_{0}+i \mu \int_{0}^{t}e^{i(t-s)\partial_x^2}|x|^{-b}(|u|^{\alpha-1}u)(s)ds :=\Lambda(u)(t).
\end{equation*}
Let $a$ and $T$ be positive real numbers (to be chosen later). 
    Define $$B(a,T):=\{u\in X'(T):\|u\|_{X'(T)}\leq a\}.$$
We will show that $\Lambda$ 
    is a contraction map on $B(a,T).$  Recall $r$ from \eqref{rvalues}.
    Firstly, we consider the linear evolution of $u_{0}$ where $u_{0}=v_{0}+w_{0}\in L^{2}+ \widehat{L}^{r}, v_{0}\in L^{2}$ and $w_{0}\in \widehat{L}^{r}.$ Assume that $T\leq 1$. Using  \eqref{st} and \eqref{spfls}, we have
    \begin{align}
        \|e^{it\partial_x^2}u_{0}\|_{X'(T)}&\leq      \|e^{it\partial_x^2}v_{0}\|_{X_{1}(T)}  +      \|e^{it\partial_x^2}w_{0}\|_{X_{2}(T)}\nonumber\\
        &\lesssim_{\alpha}      \|v_{0}\|_{L^2}  +     \|w_{0}\|_{\widehat{L}^{r}}\nonumber\\
         &\label{linearu02}\lesssim      \|u_{0}\|_{L^2  +  \widehat{L}^{r}}.
    \end{align}
    This suggests the choice of
    $
    a=C(\alpha)\|u_{0}\|_{L^2  + \widehat{L}^{r}} .   $
    Using $X_{1}(T)\hookrightarrow X'(T)$, \eqref{st} (for $r\in \{\alpha+1,2(\alpha-1),2\}$) and Lemma \ref{lemlwp2}, for $u\in B(a,T)$, we have
    \begin{align}
    \left|\left| \int_{0}^{t} e^{i(t-\tau)\partial_x^2}|x|^{-b}(|u|^{\alpha-1}u)(\tau)d\tau \right|\right|_{X'(T)}
    &\lesssim \left|\left| \int_{0}^{t} e^{i(t-\tau)\partial_x^2}|x|^{-b}(|u|^{\alpha-1}u)(\tau)d\tau \right|\right|_{X_{1}(T)} \nonumber\\
   &\lesssim_{\alpha,b} \||u|^{\alpha-1}u\|_{L^{\gamma'_{1}}_{T}L^{\rho '_{1}}(B^{c})}+\|~|x|^{-b} |u|^{\alpha-1}u\|_{L^{\gamma'_{2}}_{T}L^{\rho '_{2}}(B)}\nonumber\\
   &\lesssim ( T^{\frac{5-\alpha}{4}-\frac{1}{\nu}})\|u\|_{Y_{1}(T)}^{\alpha}\nonumber\\
    &\label{ipGamma}\lesssim ( T^{\frac{5-\alpha}{4}-\frac{1}{\nu}})\|u\|_{X'(T)}^{\alpha}.
    \end{align}
Note that $X'(T) \hookrightarrow Y_{1}(T)$ since $X_{2}(T)  \hookrightarrow Y_{1}(T)$ due to Hausdorff-Young inequality \footnote{$\widehat{L}^{r}  \hookrightarrow {L^{r}},\; 2\leq r \leq \infty$.}.  
    Taking
    \begin{equation*}
        T :=\min \left\{ 1,~C(\alpha,b) \|u_{0}\|^{-\frac{\alpha-1}{\frac{5-\alpha}{4}-\frac{1}{\nu}}}_{L^2  + \widehat{L}^{r}} \right\}~,
    \end{equation*}
 we combine \eqref{linearu02} and \eqref{ipGamma} to conclude $\Lambda(u)\in B(a, T).$ Using \eqref{eg} and the previous argument, for $u,v \in B(a, T)$, we have
\begin{align}
\hspace{-4cm}\|\Lambda(u)-\Lambda(v)\|_{X'(T)} & \lesssim \left|\left| \int_{0}^{t} e^{i(t-\tau)\partial_x^2}|x|^{-b}G(0,u,v)(\tau)d\tau \right|\right|_{X_{1}(T)} \nonumber
\end{align}
\begin{align}
&\lesssim_{\alpha,b} \|(|u|^{\alpha}+|v|^{\alpha})|u-v|\|_{L^{\gamma'_{1}}_{T}L^{\rho '_{1}}(B^{c})}+\|~|x|^{-b}(|u|^{\alpha}+|v|^{\alpha})|u-v|\|_{L^{\gamma'_{2}}_{T}L^{\rho '_{2}}(B)} \nonumber \\
    &\lesssim ( T^{\frac{5-\alpha}{4}-\frac{1}{\nu}})(\|u\|_{X'(T)}^{\alpha-1}+\|v\|_{X'(T)}^{\alpha-1})\|u-v\|_{X'(T)}.
\end{align}
By the choice of $a$ and $T$ \footnote{In the expression of $a$ and $T, C(\alpha,b)$ is chosen sufficiently small to ensure that $\Lambda(u)$ is a contraction map on $B(a, T).$}, we have $$\|\Lambda(u)-\Lambda(v)\|_{X'(T)}\leq \frac{1}{2}\|u-v\|_{X'(T)}.$$
Thus, by the contraction mapping theorem,  we obtain a unique fixed point for $\Lambda$ which is a solution to \ref{INLS}.
\end{proof}
\begin{proof}[\textbf{Proof of Theorem \ref{gwp3}}] 
 We start by decomposing (using \eqref{ipt3}) initial data $u_0 \in \widehat{L}^{p} \subset L^2 + \widehat{L}^{r} $ into two parts. Specifically, for any $N>1$  and given $u_0 \in \widehat{L}^{p},$ there exists  $\phi_0 \in L^2, \psi_0 \in \widehat{L}^{r}$ (depending on $N$)   such that 
\begin{equation}\label{dp}
    u_0= \phi_0 + \psi_0
\end{equation}
with 
\begin{eqnarray}\label{asi}
\|\phi_0\|_{L^2} \lesssim N^{\beta_{2}},  \quad \|\psi_0\|_{\widehat{L}^{r}} \lesssim  \frac{1}{N}
\end{eqnarray}
where\begin{equation}\label{betap2}
    \beta_{2} = \frac{\frac{1}{2} - \frac{1}{p}}{\frac{1}{p} - \frac{1}{r}}.
\end{equation}
The initial value problem associated with $\phi_0$, given as \eqref{ivpL2}, admits a solution described in \eqref{gwpl2}.
Let the solution of the  I.V.P. corresponding to $\psi_{0}$ be given as
$$e^{it\partial_x^2}\psi_{0}+w_{0},$$ where
$w_{0}$ can be expressed as
\begin{align*}
     w_{0} 
    =i \mu \int_{0}^{t} e^{i(t-\tau)\partial_x^2}|x|^{-b} G(v_{0} + e^{i\tau\partial_x^2}\psi_{0}, w_{0},0) \, d\tau +  i \mu \int_{0}^{t} e^{i(t-\tau)\partial_x^2}|x|^{-b} G(v_{0}, e^{i\tau\partial_x^2}\psi_{0},0) \, d\tau.
\end{align*}
Also, using \eqref{spfls},
$$e^{it\partial_x^2}\psi_{0} \in L^{\infty}(\R,\widehat{L}^{r}).$$
\begin{prop}\label{w0exist3} Let $\phi_{0} \in L^2, \psi\in \widehat{L}^{r}$ and $X_{1}(T)$ be as in  \eqref{X1(T)}. Denote by  $v_{0}$  the $L^2-$global solution  (as in \eqref{gwpl2}) for initial value $\phi_{0}$.
Then  there exists  a constant $C=C(\alpha,b)>0$ such that  integral equation $$w_{0}= i \mu \int_{0}^{t} e^{i(t-\tau)\partial_x^2}|x|^{-b} G(v_{0} + e^{i\tau\partial_x^2}\psi_{0}, w_{0},0) \, d\tau +  i \mu \int_{0}^{t} e^{i(t-\tau)\partial_x^2}|x|^{-b} G(v_{0}, e^{i\tau\partial_x^2}\psi_{0},0) \, d\tau$$ has a unique solution  $w_{0} \in X_{1}(T)$   provided $T$ satisfying
\begin{align}
    T &\leq 1 \label{c1a}\\
    T &\leq C \left( \|\phi_{0}\|_{L^{2}} + \|\psi_{0}\|_{{\widehat{L}^{r}}} \right)^{-\frac{\alpha-1}{\frac{5-\alpha}{4} - \frac{1}{\nu}}} \label{c2a}\\
   T &\leq C \left( \|\psi_{0}\|_{\widehat{L}^{r}} \right)^{-\frac{\alpha-1}{\frac{5-\alpha}{4} - \frac{1}{\nu} + \frac{\alpha-1}{q(r)}}}\label{c3a}.
\end{align}
\end{prop}
\begin{proof}
     Define $$B(A,T)=\{w\in X_{1}(T):\|w\|_{X_{1}(T)}\leq A\}$$
    with $A>0$ (to be chosen later), $T$ be the minimum of the right-hand sides of the conditions \eqref{c1a},  \eqref{c2a} and \eqref{c3a}(w.l.o.g. we may assume).
    Further, define
    $$\Gamma(w):= i \mu \int_{0}^{t} e^{i(t-\tau)\partial_x^2}|x|^{-b} G(v_{0} + e^{i\tau\partial_x^2}\psi_{0}, w_{0},0)\, d\tau +  i \mu  \int_{0}^{t} e^{i(t-\tau)\partial_x^2}|x|^{-b} G(v_{0}, e^{i\tau\partial_x^2}\psi_{0},0) \, d\tau. $$
    Firstly, we need to show that $\Gamma(w)\in B(A,T)$.\\
   Using \eqref{st} (for $r\in \{\alpha+1,2(\alpha-1),2\}$), \eqref{eg}, Lemma \ref{lemlwp2} under the assumption \eqref{c1a}, the embedding $L^{\infty}_{T} \hookrightarrow L^{q(r)}_{T}$ \footnote{$\|\cdot\|_{L^{q(r)}_{T}} \leq T^{\frac{1}{q(r)}} \|\cdot\|_{L^{\infty}_{T}}.$} and we have 
    \begin{align*}
    \hspace{-0.5cm}\left\| \int_{0}^{t} e^{i(t-\tau)\partial_x^2}|x|^{-b}G(v_{0},e^{i\tau \partial_x^2}\psi_{0},0) \, d\tau \right\|_{X_{1}(T)}
   &\lesssim_{\alpha,b}  \||v_{0}|^{\alpha-1}e^{i\tau \partial_x^2}\psi_{0} \|_{L^{\gamma'_{1}}_{T}L^{\rho'_{1}}(B^{c})}  
   \nonumber\\ 
   +  \|~|x|^{-b}|v_{0}|^{\alpha-1}e^{i\tau \partial_x^2}\psi_{0} \|_{L^{\gamma'_{2}}_{T}L^{\rho'_{2}}(B)} &  + \||e^{i\tau \partial_x^2}\psi_{0}|^{\alpha} \|_{L^{\gamma'_{1}}_{T}L^{\rho'_{1}}(B^{c})} 
   + \|~|x|^{-b}|e^{i\tau \partial_x^2}\psi_{0}|^{\alpha} \|_{L^{\gamma'_{2}}_{T}L^{\rho'_{2}}(B)} 
  \end{align*}
  \begin{align*}
   \lesssim & \left(T^{\frac{5- \alpha}{4} - \frac{1}{\nu}} \right) 
     \left(\|v_{0}\|^{\alpha-1}_{Y(T)} \|e^{i\tau \partial_x^2}\psi_{0}\|_{Y(T)}+
     \|e^{i\tau \partial_x^2}\psi_{0}\|^{\alpha }_{Y(T)}\right)  \\
   \lesssim_{\alpha} 
   & \left( T^{\frac{5 - \alpha}{4} - \frac{1}{\nu} + \frac{1}{q(r)}} \right) 
     \|v_{0}\|^{\alpha-1}_{Y(T)} \|e^{i\tau \partial_x^2}\psi_{0}\|_{L^{\infty}_{T}L^{r}} 
    + \left( T^{\frac{5 - \alpha}{4} - \frac{1}{\nu} + \frac{\alpha}{q(r)}} \right) 
     \|e^{i\tau \partial_x^2}\psi_{0}\|^{\alpha }_{L^{\infty}_{T}L^{r}}\nonumber\\
     \lesssim_{\alpha} 
     & \left( T^{\frac{5-\alpha}{4} - \frac{1}{\nu} + \frac{1}{q(r)}} \right) \| \phi_{0} \|_{L^{2}}^{\alpha-1}  \| \psi_{0}\|_{\widehat{L}^{r}}
      + \left( T^{\frac{5 - \alpha}{4} - \frac{1}{\nu} + \frac{\alpha}{q(r)}} \right) 
     \| \psi_{0}\|_{\widehat{L}^{r}}^{\alpha} \\
      = &T^{\frac{1}{q(r)}}\| \psi_{0}\|_{\widehat{L}^{r}} \left( T^{\frac{5-\alpha}{4} - \frac{1}{\nu}} \right. \| \phi_{0} \|_{L^{2}}^{\alpha-1}
      +\left. T^{\frac{5-\alpha}{4} -\frac{1}{\nu} + \frac{\alpha-1}{q(r)}}\| \psi_{0} \|_{\widehat{L}^{r}}^{\alpha}  \right) \\
     \lesssim_{\alpha,b} 
     &   T^{\frac{1}{q(r)}}\| \psi_{0} \|_{\widehat{L}^{r}}.
\end{align*}
To get the last third inequality, we have used \eqref{gwpl2} and Hausdorff-Young inequality alongwith \eqref{spfls}. The last inequality follows due to our assumptions \eqref{c2a} and \eqref{c3a}.
This suggests the choice of
    \begin{equation}\label{AA2}
    A= \frac{3}{C(\alpha,b)}T^{\frac{1}{q(r)}}\|\psi_{0}\|_{\widehat{L}^{r}}.
    \end{equation}
    where $C=C(\alpha,b)$ is the same constant as in \eqref{c2a} and \eqref{c3a}, chosen such that \begin{equation}\label{gammaw1a}
      \left|\left| \displaystyle\int_{0}^{t} e^{i(t-\tau)\partial_x^2} |x|^{-b}G(v_{0},e^{i\tau \partial_x^2}\psi_{0},0)(\tau) \, d\tau \right|\right|_{Y(T)}\leq \frac{A}{3}
    \end{equation}
    holds.
     Similarly, using \eqref{st} with $(r\in \{\alpha+1, 2(\alpha-1), 2\})$, Lemma \ref{lemlwp2}, \eqref{gwpl2} and Hausdorff-Young inequality alongwith \eqref{spfls}, we have
\begin{align*}
    \hspace{-7cm} \left\| \int_{0}^{t}  e^{i(t-\tau)\partial_x^2} |x|^{-b} G(v_{0} +  e^{i\tau \partial_x^2}\psi_{0}, w_{0},0) \, d\tau \right\|_{X_{1}(T)} 
\end{align*}
\begin{align*}
    & \lesssim_{\alpha,b} T^{\frac{5-\alpha}{4} -\frac{1}{\nu}}   \left( \|v_{0}
    + e^{i\tau \partial_x^2}\psi_{0} \|_{Y(T)}^{\alpha-1} \|w_{0}\|_{Y(T)} + \|w_{0}\|_{Y(T)}^{\alpha} \right) \\
    & \lesssim_{\alpha}  \|w_{0}\|_{Y(T)}\left\{T^{\frac{5-\alpha}{4}-\frac{1}{\nu}}   \left( \left( \|\phi_{0}\|_{L^{2}} + \|\psi_{0}\|_{\widehat{L}^{\alpha+1}} \right)^{\alpha-1} +\|w_{0}\|_{Y(T)}^{\alpha-1} \right)\right\} \\
    & \lesssim_{\alpha,b}\|w_{0}\|_{Y(T)} \left\{\frac{1}{3} +T^{\frac{5-\alpha}{4} -\frac{1}{\nu}+\frac{\alpha-1}{q(r)}}\|\psi_{0}\|^{\alpha-1}_{\widehat{L}^{r}}   \right\}.
\end{align*}
In the last inequality, we have used \eqref{c2a} in the first summand (choosing $C$ in \eqref{c2a} small enough) and substituted the norm of $w_{0}$ in $Y(T)$  by $A$, ($A$ given in \eqref{AA2}) in the second summand. Consider the second summand of the last inequality under the assumption \eqref{c3a} (choosing $C$ small enough in \eqref{c3a}) to get
        \begin{equation}\label{gamma2a}
            \left\| \int_{0}^{t} e^{i(t-\tau)\partial_x^2} |x|^{-b} G(v_{0} + e^{i\tau \partial_x^2}\psi_{0}, w_{0},0)\, d\tau \right\|_{Y(T)} \leq \frac{2A}{3}.
        \end{equation}
        Combining \eqref{gammaw1a} and \eqref{gamma2a}, we can say that $\Gamma(w_{0})$ belongs to $B(A, T).$ Contractivity of $\Gamma$ follows similarly. Thus, by the Banach fixed-point theorem, we get a unique fixed point $w$ to the integral equation \eqref{w01} on the time interval $[0, T].$
\end{proof}  
\begin{Remark}
    Using Proposition \ref{w0exist3} and \eqref{st} (estimating $w_{0}$ in $ L^{\infty}L^{2}$ in the proof of Proposition \ref{w0exist}), we have  \begin{equation*}
\|w_{0}\|_{L^{\infty}_{T}L^{2}}\lesssim_{\alpha,b} T^{\frac{1}{q(r)}}\|\psi\|_{\widehat{L}^{r}}.
    \end{equation*}
\end{Remark}
Taking \eqref{gwpl2}, Proposition \eqref{w0exist3} and \eqref{spfls} into account, the solution $u$ corresponding to $u_{0}$ lies in 
$$\{L^{\infty}_{T} L^2 ~\cap~ L^{q(r)}_{T} L^{r}\}+L^{\infty}_{T} \widehat{L}^{r}\;\left(\subset L^{q(r)}_{T} L^{r}\right).$$

    We continue with the iteration process with $k=0,1,2, \cdots K-1$. In the last iteration, let $\phi_{K}$ and $\psi_{K}$ be the initial value decomposition. As discussed in the proof of Theorem \ref{mr2}, particularly consider \eqref{secondest}, we need 
 \begin{align*}
  \|\phi_{K}\|_{L^{2}} & \leq \cdots \leq  \|\phi_{0}\|_{L^{2}}+\sum_{k=0}^{K-1}\|w_{k}\|_{L^{\infty}_{[kT,(k+1)T]}L^{2}} \nonumber \\
        & \lesssim_{\alpha,b} CN^{\beta_{2}}+T^{\frac{1}{q(r)}}\sum_{k=0}^{K-1}\|e^{ikT\partial_x^2}\psi_{0}\|_{\widehat{L}^{r}}\nonumber\\
        &\lesssim_{T^{*}} CN^{\beta_{2}}+T^{\frac{1}{q(r)}}K\frac{C}{N}
\end{align*} Thus, 
 \begin{align*}
     KT &\gtrsim_{\alpha,b,T^*} N^{1+\beta_{2}}T^{1-\frac{1}{q(r)}} \approx N^{1+\beta_{2} \left(1-\frac{(\alpha-1)\left(1-\frac{1}{q(r)}\right)}{\frac{5-\alpha}{4}-\frac{1}{\nu}}\right)}\nonumber\\
    &= N^{1-\beta_{2} \left(-1+\frac{(\alpha-1)\left(1-\frac{1}{q(r)}\right)}{\frac{5-\alpha}{4}-\frac{1}{\nu}}\right)}.
\end{align*}
Note that $N$ can be taken arbitrarily large. For any $\beta_{2}$ satisfying 
\begin{align}\label{betarange2}
0 < \beta_{2} <
\begin{cases}
 \eta \quad &\text{if}\quad  \alpha-1-\frac{\alpha-1}{q(r)}-\frac{5-\alpha}{4}+\frac{1}{\nu}>0\\
\infty \quad &\text{otherwise} 
\end{cases}
\end{align}
where
$$\eta =\frac{\frac{5-\alpha}{4}-\frac{1}{\nu}}{\alpha-1-\frac{\alpha-1}{q(r)}-\frac{5-\alpha}{4}+\frac{1}{\nu}}~,$$
the exponent of $N$ is positive in \eqref{Npower}, we get $KT>T^*$. This concludes the proof of Theorem \ref{gwp3}.
\begin{Remark}\label{etap2}
    Note that for $$ \beta_{2}=\eta=\frac{1-\frac{\alpha-1}{4}-\frac{1}{\nu}}
{ \alpha-1 -\frac{\alpha-1}{q(r)}-\frac{5-\alpha}{4}+\frac{1}{\nu}},$$ considering \eqref{betap2}, we get
$$p_{max}=\frac{(\alpha-1)(\frac{3}{2}+\frac{1}{r})}
{\frac{\alpha-1}{2}(\frac{3}{2}+\frac{1}{r})+(\frac{5-\alpha}{2}-\frac{2}{\nu})(\frac{1}{r}-\frac{1}{2})}.$$
Thus, we have $p \in (2,p_{max})$ and $p_{max}$ be given as:
     \begin{align*}p_{\max}=
         \begin{cases}
            \frac{(\alpha-1)(\frac{3}{2}+\frac{1}{r})}
{\frac{\alpha-1}{2}(\frac{3}{2}+\frac{1}{r})+(\frac{5-\alpha}{2}-\frac{2}{\nu})(\frac{1}{r}-\frac{1}{2})} &\text{if}\quad \beta = \eta \\
             r &\text{if}\quad \beta=\infty. \\
         \end{cases}
     \end{align*}
      This justifies the choice of \(p\)  in Theorem \ref{gwp3}.
\end{Remark}
\begin{Remark}\label{detailedpmax}
    Thus, when $\alpha \in (1,3)$, we have
\begin{equation*}p_{\max} := 
\begin{cases}
 \frac{3\alpha+5}{2\alpha+\frac{1}{\nu}}   \quad &\text{if} \  (\alpha-1)(\frac{3}{4}+\frac{1}{2(\alpha+1)})-\frac{5-\alpha}{4}+\frac{1}{\nu}>0 \\ 
\alpha+1 \quad  & \text{otherwise}.\end{cases}
\end{equation*}
For $\alpha \in [3,5-2b),$ we have
 \begin{equation*}p_{\max} := 
 \begin{cases}
\frac{3\alpha-2}{2(\alpha-1)+\frac{2(2-\alpha)}{\alpha-1}(1-\frac{1}{\nu})}  \quad &\text{if} \  (\alpha-1)(\frac{3}{4}+\frac{1}{4(\alpha-1)})-\frac{5-\alpha}{4}+\frac{1}{\nu}>0\\
2(\alpha-1) \quad  & \text{otherwise}.
\end{cases}
\end{equation*}
\end{Remark}
\end{proof}
{\bf Acknowledgments :}
D.~G.~B. is grateful to an ANRF-MATRICS Grant
ANRF/ARG\\M/2025/000580/MTR. D.~D. gratefully acknowledges Prof. Mathew Joseph for arranging her scientific visit to the Indian Statistical Institute, Bangalore Centre, India. 
\bibliographystyle{siam}
\bibliography{finls.bib}
\end{document}